\documentclass[10pt,a4paper]{article}
\usepackage{fullpage,mathrsfs,graphicx,framed,color,amssymb,amsmath,amsthm}
\usepackage[boxed, noline, ruled, linesnumbered]{algorithm2e}
\usepackage[colorlinks, citecolor=blue]{hyperref}
\usepackage{epsfig, graphicx}
\usepackage{latexsym,amsfonts,amsbsy,amssymb}
\usepackage{amsmath,amsthm}
\usepackage{enumerate}
\makeatletter 

\@addtoreset{equation}{section}
\makeatother 
\allowdisplaybreaks 
\newtheorem{theorem}{Theorem}[section]
\newtheorem{lemma}{Lemma}[section]
\newtheorem{corollary}{Corollary}[section]
\newtheorem{remark}{Remark}[section]

\newcommand{\comm}[1]{{\color{red}#1}}
\newcommand{\revise}[1]{{\color{blue}#1}}

\begin{document}
\title{Enhanced Error Estimates for Augmented Subspace Method\footnote{This work was supported in part
by the National Key Research and Development Program of China (2019YFA0709601), Beijing Natural Science Foundation (Z200003),
National Natural Science Foundations of China (NSFC 11771434), the National Center for Mathematics and Interdisciplinary Science, CAS.}}
\author{Haikun Dang\footnote{LSEC,
Academy of Mathematics and Systems Science,
Chinese Academy of
Sciences, No.55, Zhongguancun Donglu, Beijing 100190, China, and School of
Mathematical Sciences, University of Chinese Academy
of Sciences, Beijing, 100049
(danghaikun@lsec.cc.ac.cn)},\ \ \ 
Yifan Wang\footnote{LSEC,
Academy of Mathematics and Systems Science,
Chinese Academy of
Sciences, No.55, Zhongguancun Donglu, Beijing 100190, China, and School of
Mathematical Sciences, University of Chinese Academy
of Sciences, Beijing, 100049 (wangyifan@lsec.cc.ac.cn)},\ \ \ Hehu Xie\footnote{LSEC,
Academy of Mathematics and Systems Science,
Chinese Academy of
Sciences, No.55, Zhongguancun Donglu, Beijing 100190, China, and School of
Mathematical Sciences, University of Chinese Academy
of Sciences, Beijing, 100049 (hhxie@lsec.cc.ac.cn)} \  \  and\  \ Chenguang Zhou\footnote{LSEC, ICMSEC,
Academy of Mathematics and Systems Science, Chinese Academy of
Sciences,  Beijing 100190, China (zhouchenguang@lsec.cc.ac.cn)} }
\date{}
\maketitle
\begin{abstract}
In this paper, some enhanced error estimates are derived for the augmented subspace methods
which are designed for solving eigenvalue problems.
We will show that the augmented subspace methods have
the second order convergence rate which is better than the existing results.
These sharper estimates provide a new dependence of
convergence rate on the coarse spaces in augmented subspace methods.
These new results are also validated by some numerical examples.


\end{abstract}

\section{Introduction}
One of the fundamental problems in modern science and engineering society is to solve
large-scale eigenvalue problems. This is always a very difficult task to solve high-dimensional
eigenvalue problems which come from practical physical and chemical sciences.
Compared with linear boundary value problems, there are no so many efficient numerical
methods for solving eigenvalue problems with optimal complexity.
Solving large-scale eigenvalue problems poses significant challenges for scientific computing.
In order to solve large sparse eigenvalue problems, there have developed eigensolvers such as
Krylov subspace type methods
(Implicitly Restarted Lanczos/Arnoldi Method (IRLM/IRAM) \cite{Sorensen}),
the Preconditioned INVerse ITeration (PINVIT) method \cite{BramblePasciakKnyazev,PINVIT,Knyazev},
the Locally Optimal Block Preconditioned Conjugate Gradient (LOBPCG) method \cite{Knyazev_Lobpcg,KnyazevNeymeyr},
and the Jacobi-Davidson-type techniques \cite{Bai}.
All these popular methods include the orthogonalization steps during computing Rayleigh-Ritz problems
which are always the bottlenecks for designing efficient parallel schemes
for determining relatively many eigenpairs.
Recently, a type of multilevel correction method is proposed for solving
eigenvalue problems in \cite{full,HongXieXu,LinXie_MultiLevel,Xie_IMA,Xie_JCP,XieZhangOwhadi,XuXieZhang}.
In this multilevel correction scheme, there exists an augmented subspace which is constructed with the help of
the low dimensional finite element space defined on the coarse grid.
Based on this special augmented subspace, we have designed some efficient numerical
methods for solving eigenvalue problems and nonlinear equations. This type of augmented subspace methods
only need a low dimension finite element space on the coarse mesh and the final
finite element space on the finest mesh. This method can also work even the coarse and finest mesh has no
nested property which is an extension of the multilevel correction method.
The application of this augmented subspace can transform
the solution of the eigenvalue problem on the final level of mesh can be reduced to the
solution of boundary value problems on the final level of mesh and the solution
of the eigenvalue problem on the low dimensional augmented subspace.
The multilevel correction method and augmented subspace method give the
ways to construct the multigrid method for eigenvalue problems.
More important, we can design an eigenpair-wise parallel eigenslver
for the eigenvalue problems based on the augmented subspace.
This type of parallel method avoids doing orthogonalization and inner-products in the high
dimensional space which account for a large portion of the wall time in the parallel computation.
For more information, please refer to \cite{XuXieZhang}.

The aim of this paper is to give new and sharper error estimates for the augmented
subspace method. These new error estimates provide new investigations between
the augmented subspace method and the two-grid method \cite{XuZhou}. Roughly speaking, we will give
the following error estimate for the augmented subspace method
\begin{eqnarray*}
\|\bar u_h-u_h^{(\ell+1)}\|_a \lesssim \eta_a^2(V_H)\|\bar u_h-u_h^{(\ell)}\|_a,
\end{eqnarray*}
which is sharper  than the existed results included in
\cite{LinXie_MultiLevel,Xie_IMA,Xie_JCP,XieZhangOwhadi,XuXieZhang}.
This estimate also shows the dependence of the convergence rate for
the augmented subspace method on the low dimensional space $V_H$.

An outline of the paper goes as follows. In Section 2, we introduce the
finite element method for the eigenvalue problem and the corresponding
error estimates. The augmented subspace method and some enhanced error estimates
will be given in Section 3 which is the main part of this paper.
In section 4, two choices of the coarse spaces are discussed and some numerical examples
are provided to validate the enhanced results in this paper.
Some concluding remarks are given in the last section.

\section{Discretization by finite element method}
In this section, we introduce some notation and error estimates of
the finite element approximation for eigenvalue problems. In this
paper, the letter $C$ (with or without subscripts) denotes a generic
positive constant which may be different at different occurrences.
For convenience, the symbols $\lesssim$, $\gtrsim$ and $\approx$
will be used in this paper. That $x_1\lesssim y_1, x_2\gtrsim y_2$
and $x_3\approx y_3$, mean that $x_1\leq C_1y_1$, $x_2 \geq c_2y_2$
and $c_3x_3\leq y_3\leq C_3x_3$ for some constants $C_1, c_2, c_3$
and $C_3$ that are independent of mesh sizes.

For generality, let $V$ and $W$ denote two Hilbert spaces and $V\subset W$. Then let
$a(\cdot,\cdot)$ and $b(\cdot,\cdot)$ be two positive definite symmetric bilinear forms
on $V\times V$ and $W\times W$, respectively.
Furthermore,  based on the bilinear form $a(\cdot,\cdot)$, we can define the
norm on the space $V$ as follows
\begin{eqnarray}\label{Nomr_a}
\|v\|_a = \sqrt{a(v,v)},\ \ \ \ \forall v\in V.
\end{eqnarray}
Similarly, we can define the norm $\|\cdot\|_b$ by the bilinear form $b(\cdot,\cdot)$
on the space $W$
\begin{eqnarray}\label{Nomr_b}
\|w\|_b = \sqrt{b(w,w)},\ \ \ \ \forall w\in W.
\end{eqnarray}
In this paper, we assume that the norm $\|\cdot\|_a$ is
relatively compact with respect to the norm $\|\cdot\|_b$ \cite{Conway}.

In our methodology description, we are concerned with the following
general eigenvalue problem:
Find $(\lambda, u )\in \mathcal{R}\times V$ such that $a(u,u)=1$ and
\begin{eqnarray}\label{weak_eigenvalue_problem}
a(u,v)=\lambda b(u,v),\quad \forall v\in V.
\end{eqnarray}
It is well known that the eigenvalue problem (\ref{weak_eigenvalue_problem})
has an eigenvalue sequence $\{\lambda_j \}$ (cf. \cite{BabuskaOsborn_1989,Chatelin}):
$$0<\lambda_1\leq \lambda_2\leq\cdots\leq\lambda_k\leq\cdots,\ \ \
\lim_{k\rightarrow\infty}\lambda_k=\infty,$$ and associated
eigenfunctions
$$u_1, u_2, \cdots, u_k, \cdots,$$
where $a(u_i,u_j)=\delta_{ij}$ ($\delta_{ij}$ denotes the Kronecker function).
In the sequence $\{\lambda_j\}$, the $\lambda_j$ are repeated according to their
geometric multiplicity.

Now, let us define the finite dimensional subspace approximations of the problem
(\ref{weak_eigenvalue_problem}).
For generality, let $V_h$ denote some type of finite dimensional subspace of the Hilbert space $V$.
It is well known that the finite element method is the widest used way to build the subspace $V_h$.
For easy understanding and as an example, we use the finite element method to build the space $V_h$.
 First we generate a shape-regular triangulation $\mathcal{T}_h$ of the computing domain $\Omega\subset \mathcal{R}^d\
(d=2,3)$ into triangles or rectangles for $d=2$ (tetrahedrons or
hexahedrons for $d=3$). The diameter of a cell $K\in\mathcal{T}_h$
is denoted by $h_K$ and the mesh size $h$ describes the maximal diameter of all cells
$K\in\mathcal{T}_h$.
Based on the mesh $\mathcal{T}_h$, we can construct a finite element space denoted by
 $V_h \subset V$. For simplicity, we set $V_h$ as the Lagrange type  finite
 element space which is defined as follows
\begin{equation}\label{linear_fe_space}
  V_h = \Big\{ v_h \in C(\Omega)\ \big|\ v_h|_{K} \in \mathcal{P}_k,
  \ \ \forall K \in \mathcal{T}_h\Big\}\cap H^1_0(\Omega),
\end{equation}
where $\mathcal{P}_k$ denotes the polynomial space of degree at most $k$.

Then, we can define the standard finite element scheme for eigenvalue
problem (\ref{weak_eigenvalue_problem}):
Find $(\bar{\lambda}_h, \bar{u}_h)\in \mathcal{R}\times V_h$
such that $a(\bar{u}_h,\bar{u}_h)=1$ and
\begin{eqnarray}\label{Weak_Eigenvalue_Discrete}
a(\bar{u}_h,v_h)=\bar{\lambda}_h b(\bar{u}_h,v_h),\quad\ \  \ \forall v_h\in V_h.
\end{eqnarray}
It is well known that $V_h\subset V$ is a family of finite-dimensional spaces that satisfy
the following assumption: For any $w \in V$
\begin{eqnarray}\label{Approximation_Property}
\lim_{h\rightarrow0}\inf_{v_h\in V_h}\|w-v_h\|_a = 0.
\end{eqnarray}
From \cite{BabuskaOsborn_1989,BabuskaOsborn_Book}, the  discrete eigenvalue
problem (\ref{Weak_Eigenvalue_Discrete}) has eigenvalues:
$$0<\bar{\lambda}_{1,h}\leq \bar{\lambda}_{2,h}\leq\cdots\leq \bar{\lambda}_{k,h}
\leq\cdots\leq \bar{\lambda}_{N_h,h},$$
and corresponding eigenfunctions
\begin{eqnarray}\label{Discrete_Eigenfunctions}
\bar{u}_{1,h}, \bar{u}_{2,h}, \cdots, \bar{u}_{k,h}, \cdots, \bar{u}_{N_h,h},
\end{eqnarray}
where $a(\bar{u}_{i,h},\bar{u}_{j,h})=\delta_{ij}$, $1\leq i,j\leq N_h$ ($N_h$ is
the dimension of the finite element space $V_h$).
From the min-max principle \cite{BabuskaOsborn_1989,BabuskaOsborn_Book}, the eigenvalues
of (\ref{Weak_Eigenvalue_Discrete}) provide upper bounds for the first
$N_h$ eigenvalues of (\ref{weak_eigenvalue_problem})
\begin{eqnarray}\label{Upper_Bound_Result}
\lambda_i \leq \bar\lambda_{i,h},\ \ \ \ 1\leq i\leq N_h.
\end{eqnarray}

For the following analysis in this paper, we define $\mu_i=1/\lambda_i$ for $i=1,2,\cdots$, and
$\bar\mu_i = 1/\bar\lambda_{i,h}$ for $i=1, \cdots, N_h$.
In order to measure the error of the finite element space to the desired function, we define the following notation
\begin{eqnarray}\label{Delta_V_h}
\delta(w,V_h) = \inf_{v_h\in V_h}\|w-v_h\|_a,\ \ \ {\rm for}\ w\in V.
\end{eqnarray}
In this paper, we also need the following quantity for error analysis:
\begin{eqnarray}
\eta_a(V_h)&=&\sup_{\substack{ f\in W\\ \|f\|_b=1}}\inf_{v_h\in V_h}\|Tf-v_h\|_{a},\label{eta_a_h_Def}
\end{eqnarray}
where $T:W\rightarrow V$ is defined as
\begin{equation}\label{laplace_source_operator}
a(Tf,v) = b(f,v), \ \ \ \ \  \forall v\in V\ \ {\rm for}\  f \in W.
\end{equation}

In order to understand the method more clearly, we state the error estimate for the
eigenpair approximation by the finite element method. For this aim,
we define the finite element projection $\mathcal P_h: V\rightarrow V_h$ as follows
\begin{eqnarray}\label{Energy_Projection}
a(\mathcal P_h w, v_h) = a(w,v_h),\ \ \ \ \forall v_h\in V_h \ \ {\rm for}\  w\in V.
\end{eqnarray}
It is obvious that the finite element projection operator $\mathcal P_h$ has following error estimates.
\begin{lemma}
For any function $w\in V$, the finite element projection operator $\mathcal P_h$ has following error estimates
\begin{eqnarray}
\|w-\mathcal P_h w\|_a &= &\inf_{w_h\in V_h}\|w-w_h\|_a = \delta(w,V_h),\label{Delta_V_h_P_h}\\
\|w-\mathcal P_h w\|_b &\leq& \eta_a(V_h)\|w-\mathcal P_hw\|_a.\label{Aubin_Nitsche_Estimate}
\end{eqnarray}
\end{lemma}

Before stating error estimates of the subspace projection method, we introduce a lemma which
comes from \cite{StrangFix}. For completeness, a proof is provided here.
\begin{lemma}(\cite[Lemma 6.4]{StrangFix})\label{Strang_Lemma}
For any eigenpair $(\lambda,u)$ of (\ref{weak_eigenvalue_problem}), the following equality holds
\begin{eqnarray*}\label{Strang_Equality}
(\bar\lambda_{j,h}-\lambda)b(\mathcal P_hu,\bar u_{j,h})
=\lambda b(u-\mathcal P_hu,\bar u_{j,h}),\ \ \ j = 1, \cdots, N_h.
\end{eqnarray*}
\end{lemma}
\begin{proof}
Since $-\lambda b(\mathcal P_h u,\bar u_{j,h})$ appears on both sides, we only need to prove that
\begin{eqnarray*}
\bar\lambda_{j,h}b(\mathcal P_h u,\bar u_{j,h})=\lambda b(u,\bar u_{j,h}).
\end{eqnarray*}
From (\ref{weak_eigenvalue_problem}), (\ref{Weak_Eigenvalue_Discrete}) and (\ref{Energy_Projection}),
the following equalities hold
\begin{eqnarray*}
\bar\lambda_{j,h}b(\mathcal P_h u,\bar u_{j,h}) = a(\mathcal P_h u,\bar u_{j,h})
=a(u,\bar u_{j,h}) = \lambda b(u,\bar u_{j,h}).
\end{eqnarray*}
Then the proof is complete.
\end{proof}
The following lemma has already been presented in \cite{XieZhangOwhadi} which gives the error estimates for the one eigenpair
approximation. This lemma will be used for analyzing the error estimates for the augmented subspace method for only one eigenpair.
For the proof, please refer to \cite{XieZhangOwhadi}.
\begin{lemma}(\cite[Lemma 3.3]{XieZhangOwhadi})\label{Error_Estimate_Theorem_Old}
Let  $(\lambda,u)$ denote an exact eigenpair of the eigenvalue problem (\ref{weak_eigenvalue_problem}).
Assume the eigenpair approximation $(\bar\lambda_{i,h},\bar u_{i,h})$ has the property that
$\bar\mu_{i,h}=1/\bar\lambda_{i,h}$ is closest to $\mu=1/\lambda$.
The corresponding spectral projector $E_{i,h}: V\mapsto {\rm span}\{\bar u_{i,h}\}$
is  defined as follows
\begin{eqnarray*}
a(E_{i,h}w,\bar u_{i,h}) = a(w,\bar u_{i,h}),\ \ \ \ {\rm for}\  w\in V.
\end{eqnarray*}
Then the following error estimate holds
\begin{eqnarray}\label{Energy_Error_Estimate_Old}
\|u-E_{i,h}u\|_a&\leq& \sqrt{1+\frac{\bar\mu_{1,h}}{\delta_{\lambda,h}^2}\eta_a^2(V_h)}\|(I-\mathcal P_h)u\|_a,
\end{eqnarray}
where $\eta_a(V_h)$ is defined in (\ref{eta_a_h_Def}) and $\delta_{\lambda,h}$ is defined as follows
\begin{eqnarray}\label{Definition_Delta}
\delta_{\lambda,h} &:=& \min_{j\neq i}|\bar\mu_{j,h}-\mu|=\min_{j\neq i} \Big|\frac{1}{\bar\lambda_{j,h}}-\frac{1}{\lambda}\Big|.
\end{eqnarray}
Furthermore, the eigenvector approximation $\bar u_{i,h}$ has following
error estimate in $\|\cdot\|_b$-norm
\begin{eqnarray}\label{L2_Error_Estimate_Old}
\|u-E_{i,h}u\|_b &\leq&\Big(1+\frac{\bar\mu_{1,h}}{\delta_{\lambda,h}}\Big)\eta_a(V_h)\|u-E_{i,h}u\|_a.
\end{eqnarray}
\end{lemma}
For simplicity of notation, we assume that the eigenvalue gap $\delta_{\lambda,h}$
has a uniform lower bound which is denoted by $\delta_\lambda$ (which can be seen as the
``true" separation of the eigenvalue $\lambda$ from others) in the following parts of this paper.
This assumption is reasonable when the mesh size is small enough. We refer to
\cite[Theorem 4.6]{Saad1} and Lemma \ref{Error_Estimate_Theorem_Old} in this paper for details of the
dependence of error estimates on the eigenvalue gap.
Then we have the following simple version of the error estimates
based on Lemma \ref{Error_Estimate_Theorem_Old}.
\begin{corollary}\label{Error_Estimate_Corollary}
Under the conditions of Lemma \ref{Error_Estimate_Theorem_Old}, the following error estimates hold
\begin{eqnarray}
\|u-E_{i,h}u\|_a&\leq& \sqrt{1+\frac{1}{\lambda_1\delta_\lambda^2}\eta_a^2(V_h)}\|(I-\mathcal P_h)u\|_a,\label{Energy_Error_Estimate}\\
\|u-E_{i,h}u\|_b &\leq&\Big(1+\frac{1}{\lambda_1\delta_\lambda}\Big)\eta_a(V_h)\|u-E_{i,h}u\|_a.\label{L2_Error_Estimate}
\end{eqnarray}
\end{corollary}

In the following part of this section, we consider the error estimates for the first
$k$ eigenpair approximations associated with $\bar\lambda_{1,h}\leq \cdots\leq \bar\lambda_{k,h}$.
\begin{theorem}\label{Error_Estimate_Theorem_k}
Let us define the spectral projection $E_{k,h}: V\rightarrow {\rm span}\{\bar u_{1,h}, \cdots, \bar u_{k,h}\}$
as follows
\begin{eqnarray}
a(E_{k,h}w, \bar u_{i,h}) = a(w, \bar u_{i,h}), \ \ \ i=1, \cdots, k\ \ {\rm for}\ w\in V.
\end{eqnarray}
Then the associated exact eigenfunctions $u_1, \cdots, u_k$ of eigenvalue problem (\ref{weak_eigenvalue_problem}) have the following error estimates
\begin{eqnarray}\label{Energy_Error_Estimate_k}
\|u_i - E_{k,h} u_i\|_a \leq \sqrt{1+\frac{1}{\lambda_{k+1}\delta_{k,i,h}^2}\eta_a^2(V_h)}\|(I-\mathcal P_h)u_i\|_a,
\ \ \ \  1\leq i\leq k,
\end{eqnarray}
where 
\begin{eqnarray}\label{Definition_Delta_k_i_0}
\delta_{k,i,h} = \min_{k<j\leq N_h}\left|\frac{1}{\bar\lambda_{j,h}}-\frac{1}{\lambda_i}\right|.
\end{eqnarray}
Furthermore, these $k$ exact eigenvectors have following error estimate in $\|\cdot\|_b$-norm
\begin{eqnarray}\label{L2_Error_Estimate_k_0}
\|u_i-E_{k,h}u_i\|_b \leq \Big(1+\frac{\mu_{k+1}}{\delta_{k,i,h}}\Big)\eta_a(V_h)\|u_i-E_{k,h}u_i\|_a,
\ \ \ 1\leq i\leq k.
\end{eqnarray}
\end{theorem}
\begin{proof}
Similarly to the duality argument in the finite element method, the following inequality holds
\begin{eqnarray}\label{L2_Energy_Estiate}
&&\|(I-\mathcal P_h)u_i\|_b=\sup_{\|g\|_b=1}b((I-\mathcal P_h)u_i,g)
=\sup_{\|g\|_b=1}a((I-\mathcal P_h)u_i,Tg) \nonumber\\
&&=\sup_{\|g\|_b=1}a((I-\mathcal P_h)u_i,(I-\mathcal P_h)Tg)
\leq \eta_a(V_h)\|(I-\mathcal P_h)u_i\|_a.
\end{eqnarray}

Since $(I-E_{k,h})\mathcal P_hu_i\in V_h$ and
$(I-E_{k,h})\mathcal P_hu_i\in {\rm span}\{\bar u_{k+1,h},\cdots, \bar u_{N_h,h}\}$,
the following orthogonal expansion holds
\begin{eqnarray}\label{Orthogonal_Decomposition_k}
(I-E_{k,h})\mathcal P_hu_i=\sum_{j=k+1}^{N_h}\alpha_j\bar u_{j,h},
\end{eqnarray}
where $\alpha_j=a(\mathcal P_hu_i,\bar u_{j,h})$. From Lemma \ref{Strang_Lemma}, we have
\begin{eqnarray}\label{Alpha_Estimate}
\alpha_j&=&a(\mathcal P_hu_i,\bar u_{j,h}) = \bar\lambda_{j,h} b\big(\mathcal P_hu_i,\bar u_{j,h}\big)
=\frac{\bar\lambda_{j,h}\lambda}{\bar\lambda_{j,h}-\lambda}b\big(u_i-\mathcal P_hu_i,\bar u_{j,h}\big)\nonumber\\
&=&\frac{1}{\mu-\bar\mu_{j,h}} b\big(u_i-\mathcal P_hu_i,\bar u_{j,h}\big).
\end{eqnarray}
From the orthogonal property of eigenvectors $\bar u_{1,h},\cdots, \bar u_{m,h}$, the following equalities hold
\begin{eqnarray*}
1 = a(\bar u_{j,h},\bar u_{j,h}) = \bar\lambda_{j,h} b(\bar u_{j,h},\bar u_{j,h})= \bar\lambda_{j,h}\|\bar u_{j,h}\|_b^2,
\end{eqnarray*}
which leads to the following property
\begin{eqnarray}\label{Equality_u_j}
\|\bar u_{j,h}\|_b^2=\frac{1}{\bar\lambda_{j,h}}=\bar\mu_{j,h}.
\end{eqnarray}
From (\ref{Energy_Projection}) and definitions of eigenvectors $\bar u_{1,h},\cdots, \bar u_{N_h,h}$,
we have following equalities
\begin{eqnarray}\label{Orthonormal_Basis}
a(\bar u_{j,h},\bar u_{k,h})=\delta_{jk},
\ \ \ \ \ b\Big(\frac{\bar u_{j,h}}{\|\bar u_{j,h}\|_b},
\frac{\bar u_{k,h}}{\|\bar u_{k,h}\|_b}\Big)=\delta_{jk},\ \ \ 1\leq j,k\leq N_h.
\end{eqnarray}
Then from (\ref{Orthogonal_Decomposition_k}), (\ref{Alpha_Estimate}), (\ref{Equality_u_j}) and (\ref{Orthonormal_Basis}),
we have following estimates
\begin{eqnarray}\label{Equality_4_i}
&&\|(I-E_{k,h})\mathcal P_hu_i\|_a^2 = \Big\|\sum_{j=k+1}^{N_h}\alpha_j\bar u_{j,h}\Big\|_a^2
= \sum_{j=k+1}^{N_h}\alpha_j^2\nonumber\\
&&=\sum_{j=k+1}^{N_h} \Big(\frac{1}{\mu_i-\bar\mu_{j,h}}\Big)^2 b\big(u_i-\mathcal P_hu_i,\bar u_{j,h}\big)^2
\leq\frac{1}{\delta_{k,i,h}^2}\sum_{j=k+1}^{N_h}\|\bar u_{j,h}\|_b^2
b\Big(u_i-\mathcal P_hu_i,\frac{\bar u_{j,h}}{\|\bar u_{j,h}\|_b}\Big)^2\nonumber\\
&&=\frac{1}{\delta_{k,i,h}^2}\sum_{j=k+1}^{N_h}\bar\mu_{j,h}
b\Big(u_i-\mathcal P_hu_i,\frac{\bar u_{j,h}}{\|\bar u_{j,h}\|_b}\Big)^2\nonumber\\
&&
\leq \frac{\bar\mu_{k+1,h}}{\delta_{k,i,h}^2}\sum_{j=k+1}^{N_h}
b\Big(u_i-\mathcal P_hu_i,\frac{\bar u_{j,h}}{\|\bar u_{j,h}\|_b}\Big)^2
\leq \frac{\bar\mu_{k+1,h}}{\delta_{k,i,h}^2}\|u_i-\mathcal P_hu_i\|_b^2,
\end{eqnarray}
\revise{where the last inequality holds since $\frac{\bar u_{1,h}}{\|\bar u_{1,h}\|_b}$, $\cdots$,
$\frac{\bar u_{j,h}}{\|\bar u_{j,h}\|_b}$ are the normalorthogonal basis for the space $V_h$
in the sense of the inner product $b(\cdot, \cdot)$.}

Combining (\ref{Upper_Bound_Result}) and (\ref{Equality_4_i}) leads to the following inequality
\begin{eqnarray}\label{Equality_5_k}
\|(I-E_{k,h})\mathcal P_hu_i\|_a^2
\leq\frac{\bar\mu_{k+1,h}}{\delta_{k,i,h}^2}\eta_a(V_h)^2\|(I-\mathcal P_h)u_i\|_a^2
\leq\frac{\mu_{k+1}}{\delta_{k,i,h}^2}\eta_a(V_h)^2\|(I-\mathcal P_h)u_i\|_a^2.
\end{eqnarray}

From (\ref{Equality_5_k}) and the orthogonal property
$a(u_i-\mathcal P_hu_i, (I-E_{k,h})\mathcal P_hu_i)=0$,
we have  following error estimate
\begin{eqnarray*}
\|u_i-E_{k,h}u_i\|_a^2&=&\|u_i-\mathcal P_hu_i\|_a^2
+\|(I-E_{k,h})\mathcal P_hu_i\|_a^2\nonumber\\
&\leq&\Big(1+\frac{\mu_{k+1}}{\delta_{k,i,h}^2}\eta_a(V_h)^2\Big)
\|(I-\mathcal P_h)u_i\|_a^2.
\end{eqnarray*}
This is the desired result (\ref{Energy_Error_Estimate_k}).

Similarly, from (\ref{Upper_Bound_Result}) ,  (\ref{Orthogonal_Decomposition_k}),
(\ref{Alpha_Estimate}), (\ref{Equality_u_j}) and (\ref{Orthonormal_Basis}),
we have following estimates
\begin{eqnarray*}
&&\|(I-E_{k,h})\mathcal P_hu_i\|_b^2 = \Big\|\sum_{j=k+1}^{N_h}\alpha_j\bar u_{j,h}\Big \|_b^2
= \sum_{j=k+1}^{N_h}\alpha_j^2\|\bar u_{j,h}\|_b^2\nonumber\\
&&=\sum_{j=k+1}^{N_h} \Big(\frac{1}{\mu_i-\bar\mu_{j,h}}\Big)^2
b\big(u_i-\mathcal P_hu_i,\bar u_{j,h}\big)^2\|\bar u_{j,h}\|_b^2
\leq\frac{1}{\delta_{k,i,h}^2}\sum_{j=k+1}^{N_h}\|\bar u_{j,h}\|_b^4\
b\Big(u_i-\mathcal P_hu_i,\frac{\bar u_{j,h}}{\|\bar u_{j,h}\|_b}\Big)^2\nonumber\\
&&=\frac{1}{\delta_{k,i,h}^2}\sum_{j=k+1}^{N_h}\bar\mu_{j,h}^2
b\Big(u_i-\mathcal P_hu_i,\frac{\bar u_{j,h}}{\|\bar u_{j,h}\|_b}\Big)^2
\leq \frac{\bar\mu_{k+1,h}^2}{\delta_{k,i,h}^2}\|u_i-\mathcal P_hu_i\|_b^2
\leq \frac{\mu_{k+1}^2}{\delta_{k,i,h}^2}\|u_i-\mathcal P_hu_i\|_b^2,
\end{eqnarray*}
which leads to the inequality
\begin{eqnarray}\label{Equality_8_k}
\|(I-E_{k,h})\mathcal P_hu_i\|_b \leq \frac{\mu_{k+1}}{\delta_{k,i,h}}\|u_i-\mathcal P_hu_i\|_b.
\end{eqnarray}
From (\ref{L2_Energy_Estiate}), (\ref{Equality_8_k}) and the triangle inequality, we have
following error estimates for the eigenvector approximations in the $\|\cdot\|_b$-norm
\begin{eqnarray*}\label{Inequality_11}
&&\|u_i-E_{k,h}u_i\|_b\leq \|u_i-\mathcal P_hu_i\|_b + \|(I-E_{k,h})
\mathcal P_hu_i\|_b\nonumber\\
&&\leq\Big(1+\frac{\mu_{k+1}}{\delta_{k,i,h}}\Big)
\|(I-\mathcal P_h)u_i\|_b\leq \Big(1+\frac{\mu_{k+1}}{\delta_{k,i,h}}\Big)
\eta_a(V_h)\|(I-\mathcal P_h)u_i\|_a\nonumber\\
&&\leq \Big(1+\frac{\mu_{k+1}}{\delta_{k,i,h}}\Big)\eta_a(V_h)\|u_i-E_{k,h}u_i\|_a.
\end{eqnarray*}
This is the second desired result (\ref{L2_Error_Estimate_k}) and the proof is complete.
\end{proof}
Similarly, we assume that the eigenvalue gap $\delta_{k,i,h}$
has a uniform lower bound which is denoted by $\delta_{k,i}$ (which can be seen as the
``true" separation of the eigenvalue $\lambda_i$ from the unwanted eigenvalues)
in the following parts of this paper.
This assumption is reasonable when the mesh size is small enough.
Then we have the following simple version of the error estimates
based on Theorem \ref{Error_Estimate_Theorem_k}.
\begin{corollary}\label{Error_Estimate_Corollary_k}
Under the conditions of Theorem \ref{Error_Estimate_Theorem_k}, the following error estimates hold
\begin{eqnarray}
&&\|u_i - E_{k,h} u_i\|_a \leq \sqrt{1+\frac{1}{\lambda_{k+1}\delta_{k,i}^2}\eta_a^2(V_h)}\|(I-\mathcal P_h)u_i\|_a,
\ \ \ \  1\leq i\leq k,\label{Energy_Error_Estimate_k}\\
&&\|u_i-E_{k,h}u_i\|_b \leq \bar\eta_a(V_h)\|u_i-E_{k,h}u_i\|_a,\ \ \ 1\leq i\leq k,\label{L2_Error_Estimate_k}
\end{eqnarray}
where $\bar\eta_a(V_h)$ is defined as follows
\begin{eqnarray*}
\bar\eta_a(V_h) = \Big(1+\frac{\mu_{k+1}}{\delta_{k,i}}\Big)\eta_a(V_h).
\end{eqnarray*}
\end{corollary}

\begin{remark}
When $1\leq i\leq k$ in (\ref{Energy_Error_Estimate}), it is easy to find that the estimate (\ref{Energy_Error_Estimate_k})
is less than (\ref{Energy_Error_Estimate}) since we have the following inequalities
\begin{eqnarray*}
\frac{1}{\lambda_{k+1}} \leq \frac{1}{\lambda_1},\ \ \ \
\frac{1}{\delta_{k,i}} \leq \frac{1}{\delta_{\lambda}}.
\end{eqnarray*}
\end{remark}
From Lemma \ref{Error_Estimate_Theorem_Old}, Theorem \ref{Error_Estimate_Theorem_k} and their proofs,
we can extend the error estimates in this section to the case that the subspace is $V_{H,h}$ and
the space $V$ is replaced by $V_h$. This understanding will be used to deduce the error estimates for the augmented subspace methods in the following section.

\section{Augmented subspace method and its error estimates}\label{Section_3}
In this section, we first present the augmented subspace method for solving the eigenvalue problem (\ref{Weak_Eigenvalue_Discrete}).
This method contains solving auxiliary linear boundary value problem
in the finer finite element space $V_h$ and the eigenvalue problem on the
augmented subspace $V_{H,h}$ which is built by the coarse finite element space $V_H$
and a finite element function in the finer finite element space $V_h$.
Then, the new convergence analysis is given for this augmented subspace method.
We will find the new convergence result is sharper than the existed results in \cite{LinXie_MultiLevel,Xie_IMA,Xie_JCP,XieZhangOwhadi,XuXieZhang}.

In order to define the augmented subspace method, we first generate a coarse mesh $\mathcal{T}_H$
with the mesh size $H$ and the coarse \comm{linear} finite element space $V_H$ is
defined on the mesh $\mathcal{T}_H$.    For simplicity, in this paper, we assume the coarse space
$V_H$ is a subspace of the finite element space $V_h$ which is defined on the finer mesh $\mathcal T_h$.

For some given eigenfunction approximations $u_{1,h}^{(\ell)},\cdots, u_{k,h}^{(\ell)}$ which are
approximations for the first $k$ eigenfunctions $\bar u_{1,h},\cdots, \bar u_{k,h}$
of (\ref{Weak_Eigenvalue_Discrete}), we can do the following augmented subspace iteration step
which is defined by Algorithm \ref{Algorithm_k} to improve the accuracy of $u_{1,h}^{(\ell)},\cdots, u_{k,h}^{(\ell)}$.
\begin{algorithm}[hbt!]
\caption{Augmented subspace method for the first $k$ eigenpairs }\label{Algorithm_k}
\begin{enumerate}
\item If $\ell=1$, we define $\widehat u_{i,h}^{(\ell)}=u_{i,h}^{(\ell)}$, $i=1,\cdots, k$,   and
the augmented subspace $V_{H,h} = V_H +{\rm span}\{\widehat u_{1,h}^{(\ell)}, \cdots, \widehat u_{k,h}\}$.
Then solve the following eigenvalue problem:
Find $(\lambda_{i,h}^{(\ell)},u_{i,h}^{(\ell)})\in \mathcal{R}\times V_{H,h}$
such that $a(u_{i,h}^{(\ell)},u_{i,h}^{(\ell)})=1$ and
\begin{equation}\label{parallel_correct_eig_exact_1}
a(u_{i,h}^{(\ell)},v_{H,h}) = \lambda_{i,h}^{(\ell)}b(u_{i,h}^{(\ell)},v_{H,h}),
\ \ \ \ \ \forall v_{H,h}\in V_{H,h},\ \ \ i=1, \cdots, k.
\end{equation}

\item Solve the following linear boundary value problems:
Find $\widehat{u}_{i,h}^{(\ell+1)}\in V_h$ such that
\begin{equation}\label{Linear_Equation_k}
a(\widehat{u}_{i,h}^{(\ell+1)},v_h) = \lambda_{i,h}^{(\ell)}b(u_{i,h}^{(\ell)},v_h),
\ \  \forall v_h\in V_h,\ \ \ i=1, \cdots, k.
\end{equation}

\item Define the augmented subspace $V_{H,h} = V_H +
{\rm span}\{\widehat{u}_{1,h}^{(\ell+1)}, \cdots, \widehat u_{k,h}^{(\ell+1)}\}$ and solve the following eigenvalue problem:
Find $(\lambda_{i,h}^{(\ell+1)},u_{i,h}^{(\ell+1)})\in \mathcal{R}\times V_{H,h}$
such that $a(u_{i,h}^{(\ell+1)},u_{i,h}^{(\ell+1)})=1$ and
\begin{equation}\label{Aug_Eigenvalue_Problem_k}
a(u_{i,h}^{(\ell+1)},v_{H,h}) = \lambda_{i,h}^{(\ell+1)}b(u_{i,h}^{(\ell+1)},v_{H,h}),
\ \ \ \ \ \forall v_{H,h}\in V_{H,h},\ \ \ i=1, \cdots, k.
\end{equation}
Solve (\ref{Aug_Eigenvalue_Problem_k})  to obtain $(\lambda_{1,h}^{(\ell+1)},u_{1,h}^{(\ell+1)}), \cdots,
(\lambda_{k,h}^{(\ell+1)},u_{k,h}^{(\ell+1)})$.
\item Set $\ell=\ell+1$ and go to Step 2 for the next iteration until convergence.
\end{enumerate}
\end{algorithm}
\begin{theorem}\label{Theorem_Error_Estimate_k}
Let us define the spectral projection $E_{k,h}^{(\ell+1)}: V\rightarrow {\rm span}\{\bar u_{1,h}^{(\ell+1)}, \cdots, \bar u_{k,h}^{(\ell+1)}\}$ for any integer $\ell\geq 0$
as follows
\begin{eqnarray}
a(E_{k,h}^{(\ell+1)}w, u_{i,h}^{(\ell+1)}) = a(w, u_{i,h}^{(\ell+1)}), \ \ \ i=1, \cdots, k\ \ {\rm for}\ w\in V.
\end{eqnarray}
There exist exact eigenfunction $\bar u_{1,h},\cdots, \bar u_{k,h}$ of (\ref{Weak_Eigenvalue_Discrete}) such that the resultant eigenfunction
approximations $u_1^{(\ell+1)},\cdots, u_k^{(\ell+1)}$ have the following error estimate
\begin{eqnarray}\label{Error_Estimate_Inverse}
\big\|\bar u_{i,h} -E_{k,h}^{(\ell+1)}u_{i,h} \big\|_a \leq
\bar\lambda_{i,h} \sqrt{1+\frac{\eta_a^2(V_H)}{\lambda_{k+1}\big(\delta_{k,k}^{(\ell+1)}\big)^2}}
\left(1+\frac{\mu_{k+1}}{\delta_{k,i}^{(\ell+1)}}\right)\eta_a^2(V_H)\big\|\bar u_{i,h} - E_{k,h}^{(\ell)}u_{i,h} \big\|_a.
\end{eqnarray}
Furthermore, the following $\|\cdot\|_b$-norm error estimate hold
\begin{eqnarray}\label{L2_Error_Estimate_Algorithm_1}
&&\|\bar u_{i,h} -E_{k,h}^{(\ell+1)}\bar u_{i,h} \|_b\leq \left(1+\frac{\mu_{k+1}}{\delta_{k,i}^{(\ell+1)}}\right)\eta_a(V_H) \|\bar u_{i,h} -E_{k,h}^{(\ell+1)}\bar u_{i,h}\|_a.
\end{eqnarray}
\end{theorem}
\begin{proof}
First, let us consider the error estimate for the initial approximations $u_{1,h}^{(1)}, \cdots, u_{k,h}^{(1)}$.
From Corollary \ref{Error_Estimate_Corollary_k}, there exist exact eigenvectors $\bar u_{1,h},\cdots, \bar u_{k,h}$ such that the following error estimates for the eigenvector approximations
$u_{1,h}^{(1)},\cdots,  u_{k,h}^{(1)}$ hold for $i=1, \cdots, k$
\begin{eqnarray*}
\|\bar u_{i,h}- E_{k,h}^{(1)}\bar u_{i,h}\|_a
&\leq& \sqrt{1+\frac{\eta_a^2(V_{H,h})}{\lambda_{k+1}\big(\delta_{k,i}^{(1)}\big)^2}}
\|(I-\mathcal P_{H,h})\bar u_{i,h}\|_a\nonumber\\
&\leq&   \sqrt{1+\frac{\eta_a^2(V_H)}{\lambda_{k+1}\big(\delta_{k,i}^{(1)}\big)^2}}
\|(I-\mathcal P_{H,h})\bar u_{i,h}\|_a,
\end{eqnarray*}
and
\begin{eqnarray}\label{Inequality_13}
&&\|\bar u_{i,h}- E_{k,h}^{(1)}\bar u_{i,h}\|_b\leq \left(1+\frac{\mu_{k+1}}{\delta_{k,i}^{(1)}}\right)\eta_a(V_H)
\|\bar u_{i,h}- E_{k,h}^{(1)}\bar u_{i,h}\|_a,
\end{eqnarray}
where we have used the inequality $\eta_a(V_{H,h})\leq \eta_a(V_H)$ since $V_H\subset V_{H,h}$.

Then the result (\ref{L2_Error_Estimate_Algorithm_1})
holds  for $\ell=1$. Here the induction method is adopted to prove that
(\ref{Error_Estimate_Inverse}) and (\ref{L2_Error_Estimate_Algorithm_1})
hold for any $\ell\geq 1$. For this aim, we assume the estimates
(\ref{Error_Estimate_Inverse}) and (\ref{L2_Error_Estimate_Algorithm_1})
holds for $\ell-1$. Then let us prove that they also hold for $\ell$ based on this assumption.

From Algorithm \ref{Algorithm_k}, it is easy to know that $u_{1,h}^{(\ell)}, \cdots, u_{k,h}^{(\ell)}$
is the orthogonal basis for the space
${\rm span}\{u_{1,h}^{(\ell)}, \cdots, u_{k,h}^{(\ell)}\}$. We define the $b(\cdot,\cdot)$-orthogonal
projection operator $\pi_{k,h}^{(\ell)}$ to the space ${\rm span}\{u_{1,h}^{(\ell)}$, $\cdots$, $u_{k,h}^{(\ell)}\}$.
Then there exist $k$ real numbers $q_1, \cdots, q_k \in \mathbb R$ such that $\pi_{k,h}^{(\ell)}\bar u_{i,h}$ has following expansion
\begin{eqnarray}\label{Expansion_L2}
\pi_{k,h}^{(\ell)}\bar u_{i,h} = \sum_{j=1}^k q_ju_{j,h}^{(\ell)}.
\end{eqnarray}
From the orthogonal property of the projection operator $\mathcal P_{H,h}$, (\ref{L2_Error_Estimate_k}),  (\ref{Linear_Equation_k}), (\ref{Inequality_13}),
(\ref{Expansion_L2}) and induction assumption, the following inequalities hold
\begin{eqnarray}\label{Inequality_16}
&&\|\bar u_{i,h} - \mathcal P_{H,h}\bar u_{i,h}\|_a^2 =
a(\bar u_{i,h} - \mathcal P_{H,h}\bar u_{i,h}, \bar u_{i,h} - \mathcal P_{H,h}\bar u_{i,h})\nonumber\\
&&=a\Big(\bar u_{i,h} - \sum_{j=1}^k\bar\lambda_{i,h}\frac{q_j}{\lambda_{j,h}^{(\ell)}}
\widehat u_{j,h}^{(\ell+1)}, \bar u_{i,h} - \mathcal P_{H,h}\bar u_{i,h}\Big)\nonumber\\
&&=\bar\lambda_{i,h} b\Big(\bar u_{i,h} - \sum_{j=1}^k\frac{q_j}{\lambda_{j,h}^{(\ell)}}\lambda_{j,h}^{(\ell)}u_{j,h}^{(\ell)}, \bar u_{i,h} - \mathcal P_{H,h}\bar u_{i,h}\Big)\nonumber\\
&&=\bar\lambda_{i,h} b\Big(\bar u_{i,h} - \sum_{j=1}^kq_ju_{j,h}^{(\ell)}, \bar u_{i,h} - \mathcal P_{H,h}\bar u_{i,h}\Big)
=\bar\lambda_{i,h} b\Big(\bar u_{i,h} - \pi_{k,h}^{(\ell)}\bar u_{i,h}, \bar u_{i,h} - \mathcal P_{H,h}\bar u_{i,h}\Big)\nonumber\\
&&\leq \bar\lambda_{i,h}\big\|\bar u_{i,h} - \pi_{k,h}^{(\ell)}\bar u_{i,h}\big\|_b\big\|\bar u_{i,h} - \mathcal P_{H,h}\bar u_{i,h}\big\|_b\nonumber\\
&&\leq \bar\lambda_{i,h}\big\|\bar u_{i,h} - E_{k,h}^{(\ell)}\bar u_{i,h}\big\|_b\big\|\bar u_{i,h} - \mathcal P_{H,h}\bar u_{i,h}\big\|_b\nonumber\\
&&\leq  \bar\lambda_{i,h} \bar\eta_a(V_{H})\big\|\bar u_{i,h}- E_{k,h}^{(\ell)}\bar u_{i,h}\big\|_a \eta_a(V_H)\big\|\bar u_{i,h} - \mathcal P_{H,h}\bar u_{i,h}\big\|_a.
\end{eqnarray}
Then from (\ref{Inequality_16}), we have the following estimate
\begin{eqnarray}\label{Inequality_18_k}
\Big\|\bar u_{i,h} - \mathcal P_{H,h}\bar u_{i,h}\Big\|_a \leq \bar\lambda_{i,h} \bar\eta_a(V_{H})
\eta_a(V_H)\big\|\bar u_{i,h}- E_{k,h}^{(\ell)}\bar u_{i,h}\big\|_a.
\end{eqnarray}
Combining Corollary \ref{Error_Estimate_Corollary_k} and (\ref{Inequality_18_k}) leads to the following estimate
\begin{eqnarray}\label{Inequality_19}
\big\|\bar u_{i,h} -E_{k,h}^{(\ell+1)}\bar u_{i,h}\big\|_a \leq \bar\lambda_{i,h} \sqrt{1+\frac{\eta_a^2(V_H)}{\lambda_{k+1}
\big(\delta_{k,i}^{(\ell+1)}\big)^2}} \bar\eta_a(V_{H})
\eta_a(V_H)\big\|\bar u_{i,h}- E_{k,h}^{(\ell)}\bar u_{i,h}\big\|_a.
\end{eqnarray}
Similarly to the proof of Theorem \ref{Error_Estimate_Theorem_k}, we have the following $\|\cdot\|_b$-error estimate
\begin{eqnarray}\label{Inequality_17}
&&\|\bar u_{i,h}-E_{k,h}^{(\ell+1)}\bar u_{i,h}\|_b\leq \left(1+\frac{\mu_{k+1}}{\delta_{k,i}^{(\ell+1)}}\right)
\eta_a(V_H)\|\bar u_{i,h}-E_{k,h}^{(\ell+1)}\bar u_{i,h}\|_a.
\end{eqnarray}
From (\ref{Inequality_19}) and (\ref{Inequality_17}), we know that the estimates (\ref{Error_Estimate_Inverse})
and (\ref{L2_Error_Estimate_Algorithm_1}) holds for the integer $\ell$. Then the proof is complete.
\end{proof}
\begin{remark}
From the convergence result (\ref{Error_Estimate_Inverse}) in Theorem \ref{Theorem_Error_Estimate_k},
in order to accelerate the convergence rate, we should
decrease the term $\eta_a(V_H)$ which depends on the coarse space $V_H$.
Then enlarging the subspace $V_H$ can accelerate the convergence.
\end{remark}
\begin{remark}\label{Remark_Eigenvalue}
In this paper, we are only concerned with the error estimates for the eigenvector
approximation since the error estimates for the eigenvalue approximation
can be easily deduced from the following error expansion
\begin{eqnarray*}\label{rayexpan}
0\leq \widehat{\lambda}_i-\bar\lambda_{i,h}
=\frac{\big(A(\bar u_{i,h}-\psi),\bar u_{i,h}-\psi\big)}{(\psi,\psi)}-\bar\lambda_{i,h}
\frac{\big(\bar u_{i,h}-\psi,\bar u_{i,h}-\psi\big)}{(\psi,\psi)}\leq \frac{\|\bar u_{i,h}-\psi\|_a^2}{\|\psi\|_b^2},
\end{eqnarray*}
where $\psi$ is the eigenvector approximation for the exact eigenvector $\bar u_{i,h}$ and
\begin{eqnarray*}
\widehat{\lambda}_i=\frac{(A\psi,\psi)}{(\psi,\psi)}.
\end{eqnarray*}
\end{remark}
It is obvious that the parallel computing method can be used for Step 2 of Algorithm \ref{Algorithm_k}
since each linear equation can be solved independently. Furthermore, the augmented subspace method
can be used to design a complete parallel scheme for eigenvalue problems.
For this aim,  we give another version of the augmented subspace method for only one (may be not the smallest one) eigenpair.
The corresponding numerical method is defined by Algorithm \ref{Algorithm_1}.
This idea has already been proposed and analyzed in \cite{XuXieZhang}.
But, we will give a sharper error estimate for this type of method.

In this section,  we assume the given eigenpair
approximation $(\lambda_h^{(\ell)}, u_h^{(\ell)})\in\mathcal R\times V_h$  with different superscript is closet
to an exact eigenpair $(\bar\lambda_h, \bar u_h)$ of (\ref{Weak_Eigenvalue_Discrete}).
Based on these settings, we can give the following convergence result for the augmented
subspace method defined by Algorithm \ref{Algorithm_1}.
\begin{algorithm}[hbt!]
\caption{Augmented subspace method for one eigenpair}\label{Algorithm_1}
\begin{enumerate}
\item If $\ell=1$, we define $\widehat u_h^{(\ell)}=u_h^{(\ell)}$  and
the augmented subspace $V_{H,h} = V_H +{\rm span}\{\widehat u_h^{(\ell)}\}$.
Then solve the following eigenvalue problem:
Find $(\lambda_h^{(\ell)},u_h^{(\ell)})\in \mathcal{R}\times V_{H,h}$
such that $a(u_h^{(\ell)},u_h^{(\ell)})=1$ and
\begin{equation}\label{parallel_correct_eig_exact}
a(u_h^{(\ell)},v_{H,h}) = \lambda_h^{(\ell)}b(u_h^{(\ell)},v_{H,h}),
\ \ \ \ \ \forall v_{H,h}\in V_{H,h}.
\end{equation}

\item Solve the following linear boundary value problem:
Find $\widehat{u}_h^{(\ell+1)}\in V_h$ such that
\begin{equation}\label{Linear_Equation}
a(\widehat{u}_h^{(\ell+1)},v_h) = \lambda_{h}^{(\ell)}b(u_{h}^{(\ell)},v_h),
\ \  \forall v_h\in V_h.
\end{equation}
\item Define the augmented subspace $V_{H,h} = V_H +
{\rm span}\{\widehat{u}_h^{(\ell+1)}\}$ and solve the following eigenvalue problem:
Find $(\lambda_h^{(\ell+1)},u_h^{(\ell+1)})\in \mathcal{R}\times V_{H,h}$
such that $a(u_h^{(\ell+1)},u_h^{(\ell+1)})=1$ and
\begin{equation}\label{parallel_correct_eig_exact}
a(u_h^{(\ell+1)},v_{H,h}) = \lambda_h^{(\ell+1)}b(u_h^{(\ell+1)},v_{H,h}),
\ \ \ \ \ \forall v_{H,h}\in V_{H,h}.
\end{equation}
Solve (\ref{parallel_correct_eig_exact}) and the output $(\lambda_h^{(\ell+1)},u_h^{(\ell+1)})$
is chosen such that $u_h^{(\ell+1)}$ has the largest component in ${\rm span}\{\ \widehat{u}_h^{(\ell+1)}\}$
among all eigenfunctions of (\ref{parallel_correct_eig_exact}).
\item Set $\ell=\ell+1$ and go to Step 2 for the next iteration until convergence.
\end{enumerate}
\end{algorithm}

\begin{theorem}\label{Theorem_Error_Estimate_1}
For $\ell\geq 1$, according to the eigenpair approximation $(\lambda_h^{(\ell)},u_h^{(\ell)})\in\mathcal R\times V_h$,
we define the spectral projectors $E_h^{(\ell)}: V\mapsto {\rm span}\{u_h^{(\ell)}\}$ as follows
\begin{eqnarray*}
a(E_h^{(\ell)}w, u_h^{(\ell)}) = a(w, u_h^{(\ell)}),\ \ \ \ {\rm for}\  w\in V.
\end{eqnarray*}
Then the eigenpair approximation $(\lambda_h^{(\ell+1)},u_h^{(\ell+1)})\in\mathcal R\times V_h$ produced by
Algorithm \ref{Algorithm_1} satisfies the following error estimates
\begin{eqnarray}
\|\bar u_h-E_h^{(\ell+1)}\bar u_h\|_a &\leq&\bar\lambda_h \sqrt{1+\frac{\eta_a^2(V_H)}{\lambda_1\delta_\lambda^2}}
\Big(1+\frac{1}{\lambda_1\delta_\lambda}\Big)\eta_a^2(V_H)\|\bar u_h-E_h^{(\ell)}\bar u_h\|_a,\label{Estimate_h_1_a}\\
\|\bar u_h-E_h^{(\ell+1)}\bar u_h\|_b&\leq& \Big(1+\frac{1}{\lambda_1\delta_\lambda}\Big)\eta_a(V_H)
\|\bar u_h-E_h^{(\ell+1)}\bar u_h\|_a.\label{Estimate_h_1_b}
\end{eqnarray}
\end{theorem}
\begin{proof}
First, let us consider the error estimate for the initial approximations $u_h^{(1)}$.
From Corollary \ref{Error_Estimate_Corollary},  there exist exact eigenfunction $\bar u_h$ of (\ref{Weak_Eigenvalue_Discrete})  such that
the following error estimates hold for the eigenvector approximation $u_h^{(1)}$
\begin{eqnarray*}
\|\bar u_h- E_h^{(1)}\bar u_h\|_a
&\leq& \sqrt{1+\frac{\eta_a^2(V_{H,h})}{\lambda_1\delta_\lambda^2}}
\|(I-\mathcal P_{H,h})\bar u_h\|_a\nonumber\\
&\leq&  \sqrt{1+\frac{\eta_a^2(V_H)}{\lambda_1\delta_\lambda^2}}
\|(I-\mathcal P_{H,h})\bar u_h\|_a,
\end{eqnarray*}
and
\begin{eqnarray}\label{Inequality_131}
\|\bar u_h- E_h^{(1)}\bar u_h\|_b&\leq& \Big(1+\frac{1}{\lambda_1\delta_\lambda}\Big)\eta_a(V_{H,h})
\|\bar u_h- E_h^{(1)}\bar u_h\|_a\nonumber\\
&\leq&\Big(1+\frac{1}{\lambda_1\delta_\lambda}\Big)\eta_a(V_H)
\|\bar u_h- E_h^{(1)}\bar u_h\|_a,
\end{eqnarray}
where we have used the inequality $\eta_a(V_{H,h})\leq \eta_a(V_H)$ since $V_H\subset V_{H,h}$.

Then the result (\ref{Estimate_h_1_b}) holds  for $\ell=1$. Here the induction method is adopted to prove that
(\ref{Estimate_h_1_a}) and (\ref{Estimate_h_1_b})
hold for any $\ell\geq 1$. For this aim, we assume the estimates
(\ref{Estimate_h_1_a}) and (\ref{Estimate_h_1_b})
holds for $\ell-1$. Then let us prove that they also hold for $\ell$ based on this assumption.

We define the $b(\cdot,\cdot)$-orthogonal projection operator
$\pi_h^{(\ell)}$ to the space ${\rm span}\{u_h^{(\ell)}\}$.
Then there exists a real number $q\in\mathcal  R$
such that $\pi_h^{(\ell)}\bar u_h = q u_h^{(\ell)}$.
Then from  the orthogonal property of the projection operator $\mathcal P_{H,h}$, (\ref{Aubin_Nitsche_Estimate}),
(\ref{Inequality_13}), (\ref{Linear_Equation}) and
the induction assumption, the following inequalities hold
\begin{eqnarray}\label{Inequality_16_2}
&&\|\bar u_h - \mathcal P_{H,h}\bar u_h\|_a^2
= a(\bar u_h - \mathcal P_{H,h}\bar u_h, \bar u_h - \mathcal P_{H,h}\bar u_h)\nonumber\\
&&=a\left(\bar u_h -\frac{\bar\lambda_h}{\lambda_h^{(\ell)}}q\widehat u_h^{(\ell+1)}, \bar u_h - \mathcal P_{H,h}\bar u_h\right)\nonumber\\
&&=a\big(\bar u_h,  \bar u_h - \mathcal P_{H,h}\bar u_h\big) - \frac{\bar\lambda_h}{\lambda_h^{(\ell)}}q
a\big(\widehat u_h^{(\ell+1)}, \bar u_h - \mathcal P_{H,h}\bar u_h\big)\nonumber\\
&&=\bar\lambda_h b(\bar u_h,  \bar u_h - \mathcal P_{H,h}\bar u_h) - \bar\lambda_h
b(q u_h^{(\ell)}, \bar u_h - \mathcal P_{H,h}\bar u_h)\nonumber\\
&&=\bar\lambda_h b(\bar u_h - \pi_h^{(\ell)}\bar u_h, \bar u_h - \mathcal P_{H,h}\bar u_h)\nonumber\\
&&\leq \bar\lambda_h\|\bar u_h - \pi_h^{(\ell)}\bar u_h\|_b\|\bar u_h - \mathcal P_{H,h}\bar u_h\|_b
\leq \bar\lambda_h\|\bar u_h - E_h^{(\ell)}\bar u_h\|_b\|\bar u_h - \mathcal P_{H,h}\bar u_h\|_b\nonumber\\
&&\leq  \bar\lambda_h\Big(1+\frac{1}{\lambda_1\delta_\lambda}\Big) \eta_a(V_{H,h})\|\bar u_h- E_h^{(\ell)}\bar u_h\|_a
\eta_a(V_{H,h})\|\bar u_h - \mathcal P_{H,h}\bar u_h\|_a\nonumber\\
&&\leq  \bar\lambda_h \Big(1+\frac{1}{\lambda_1\delta_\lambda}\Big)
\eta_a^2(V_H)\|\bar u_h- E_h^{(\ell)}\bar u_h\|_a\|\bar u_h - \mathcal P_{H,h}\bar u_h\|_a,
\end{eqnarray}
where we also used the inequality $\eta_a(V_{H,h})\leq \eta_a(V_H)$ since $V_H\subset V_{H,h}$.

From (\ref{Inequality_16_2}), we have the following estimate
\begin{eqnarray}\label{Inequality_17_2}
\|\bar u_h - \mathcal P_{H,h}\bar u_h\|_a \leq \bar\lambda_h \Big(1+\frac{1}{\lambda_1\delta_\lambda}\Big)\eta_a^2(V_H)\|\bar u_h- E_h^{(\ell)}\bar u_h\|_a.
\end{eqnarray}
Combining Lemma \ref{Error_Estimate_Theorem_Old}, Corollary \ref{Error_Estimate_Corollary}
and (\ref{Inequality_17_2}), we have the following estimate
\begin{eqnarray}\label{Inequality_18}
\|\bar u_h-E_h^{(\ell+1)}\bar u_h\|_a \leq \bar\lambda_h
\sqrt{1+\frac{\eta_a^2(V_H)}{\lambda_1\delta_\lambda^2}}\Big(1+\frac{1}{\lambda_1\delta_\lambda}\Big) \eta_a^2(V_H)\|\bar u_h-u_h^{(\ell)}\|_a.
\end{eqnarray}
Similarly to the proof of Lemma \ref{Error_Estimate_Theorem_Old}, the following $\|\cdot\|_b$-error estimate hold
\begin{eqnarray}\label{Inequality_19_2}
&&\|\bar u_h- E_h^{(\ell+1)}\bar u_h\|_b\leq \Big(1+\frac{1}{\lambda_1\delta_\lambda}\Big) \eta_a(V_H)\|\bar u_h- E_h^{(\ell+1)}\bar u_h\|_a.
\end{eqnarray}
From (\ref{Inequality_18}) and (\ref{Inequality_19_2}), we know that the estimates (\ref{Estimate_h_1_a})
and (\ref{Estimate_h_1_b}) also holds for $\ell$. Then the proof is complete.
\end{proof}
\begin{corollary}
Under the conditions of Theorem \ref{Theorem_Error_Estimate_1}, the eigenfunction approximation $u_h^{(\ell+1)}$ has following error estimates
\begin{eqnarray}
\|\bar u_h-E_h^{(\ell+1)}\bar u_h\|_a &\leq&\big(\gamma(\bar\lambda_h)\big)^\ell\  \|\bar u_h-E_h^{(1)}\bar u_h\|_a,\label{Estimate_h_1_a_ell}\\
\|\bar u_h-E_h^{(\ell+1)}\bar u_h\|_b&\leq& \Big(1+\frac{1}{\lambda_1\delta_\lambda}\Big)\eta_a(V_H)
\|\bar u_h-E_h^{(\ell+1)}\bar u_h\|_a,\label{Estimate_h_1_b_ell}
\end{eqnarray}
where
\begin{eqnarray}\label{Definition_Gamma}
\gamma(\bar\lambda_h) = \bar\lambda_h \sqrt{1+\frac{\eta_a^2(V_H)}{\lambda_1\delta_\lambda^2}}
\Big(1+\frac{1}{\lambda_1\delta_\lambda}\Big)\eta_a^2(V_H).
\end{eqnarray}
\end{corollary}
The error estimate for the eigenvalue approximations $\lambda_h^{(\ell)}$ can be deduced from
Theorem \ref{Theorem_Error_Estimate_1}  and Remark \ref{Remark_Eigenvalue}.

\section{The application to second order elliptic eigenvalue problem}
In this section, we will show the applications of augmented subspace
methods to the second order elliptic eigenvalue problem. These numerical schemes can improve
the efficiency for solving the eigenvalue problems. Especially, based on the property of the
augmented subspace method, the choice of the coarse finite element space $V_H$ is independent of the
finest finite element space.

Here, we are concerned with the second order elliptic eigenvalue problem, i.e., in (\ref{weak_eigenvalue_problem}),
the bilinear forms $a(\cdot, \cdot)$ and $b(\cdot,\cdot)$ are defined as follows
\begin{eqnarray*}
a(u,v)=\int_{\Omega}\nabla u\cdot\mathcal{A}\nabla vd\Omega, &&
b(u,v)=\int_{\Omega}\rho uvd\Omega,
\end{eqnarray*}
where $\Omega\subset\mathcal{R}^d \ (d=2,3)$ is a bounded domain,
$\mathcal{A}\in \big(W^{1,\infty}(\Omega)\big)^{d\times d}$  a
uniformly positive definite matrix on $\Omega$ and $\rho\in
W^{0,\infty}(\Omega)$ is a uniformly positive function on $\Omega$.
We pose homogeneous Dirichlet boundary condition to the problem and it means
here $V=H_0^1(\Omega)$ and $W=L^2(\Omega)$ (cf. \cite{Adams}).
In order to use the finite element discretization method, we employ the meshes defined
in section 2.

Here the augmented subspace methods defined by Algorithms \ref{Algorithm_k} and \ref{Algorithm_1}
are applied to the second order elliptic eigenvalue problem.
The main ingredient is to discuss the way to construct the coarse coarse space $V_H$
based on the fine space $V_h$. There have two obvious ways to produce the coarse space $V_H$.
In the first way, the coarse space $V_H$ and fine space $V_h$ are defined on the same mesh
denoted by $\mathcal T_H$ in this section.
But the degree of the fine space $\mathcal T_h$ is higher than that of the coarse space $V_H$.
This means the coarse space $V_H$ is chosen as the linear finite element space.
The second way to produce the coarse space is based on the two-grid idea from \cite{XuZhou}.
In this way, the coarse space $V_H$ is defined on the coarse grid $\mathcal T_H$ but
the fine space $V_h$ is defined on the finer grid $\mathcal T_h$.
In these two ways, the coarse space $V_H$ are both chosen as the linear finite element space on the
mesh $\mathcal T_H$,  we have the following estimate for the quantity $\eta_a(V_H)$ (cf. \cite{BrennerScott,Ciarlet})
\begin{eqnarray}
\eta_a(V_H)\leq CH,
\end{eqnarray}
where the constant depends on the matrix $\mathcal A$, scalar $\rho$ and the
shape of the mesh $\mathcal T_H$.

Based on Theorems \ref{Theorem_Error_Estimate_k} and \ref{Theorem_Error_Estimate_1},
the convergence result can be concluded with the following inequalities
\begin{eqnarray}
&&\big\|\bar u_{i,h} -E_{k,h}^{(\ell+1)}u_{i,h} \big\|_a \leq C\big(CH\big)^{2\ell}\big\|\bar u_{i,h} - E_{k,h}^{(1)}u_{i,h} \big\|_a,\label{Test_1_1}\\
&&\|\bar u_{i,h} -E_{k,h}^{(\ell+1)}\bar u_{i,h} \|_b\leq CH \|\bar u_{i,h} -E_{k,h}^{(\ell+1)}\bar u_{i,h}\|_a,\label{Test_1_0}
\end{eqnarray}
and
\begin{eqnarray}
\|\bar u_h-E_h^{(\ell+1)}\bar u_h\|_a &\leq&\big(CH\big)^{2\ell}\  \|\bar u_h-E_h^{(1)}\bar u_h\|_a,\label{Test_2_1}\\
\|\bar u_h-E_h^{(\ell+1)}\bar u_h\|_b&\leq& CH\big\|\bar u_h-E_h^{(\ell+1)}\bar u_h\|_a.\label{Test_2_0}
\end{eqnarray}
The aim of this section is to check these convergence results by some numerical examples.
In these numerical experiments, Algorithms \ref{Algorithm_k} and \ref{Algorithm_1} are implemented
for solving the following standard Laplace eigenvalue problem: Find $(\lambda,u)\in\mathcal R\times H_0^1(\Omega)$ such that
\begin{eqnarray}\label{Test_Problem}
\left\{
\begin{array}{rcl}
-\Delta u&=& \lambda u,\ \ \ {\rm in}\ \Omega,\\
u&=&0,\ \ \ \ {\rm on}\ \partial\Omega,\\
\|u\|_1^2&=&1,
\end{array}
\right.
\end{eqnarray}
where the computing domain is set to be the unit square $\Omega=(0,1)\times (0,1)$.

In all numerical testes, the initial eigenfunction approximation is produced by solving
 the eigenvalue problem (\ref{Test_Problem}) on the coarse space $V_H$.
 The exact finite element eigenfunction $\bar u_h$ is obtained by solving the eigenvalue
problem directly on the fine space $V_h$.

\subsection{Augmented subspace by low order finite element space}
In the first subsection, we check the convergence results (\ref{Test_1_1})-(\ref{Test_2_0})
for the fine space is chosen as the high order finite element space.
In these tests, the initial eigenfunction approximation is produced by solving the eigenvalue
problems on the coarse space $V_H$. Then we do the iteration steps by the augmented subspace
method defined by Algorithms \ref{Algorithm_k} and \ref{Algorithm_1}.

In the first way,  the spaces $V_H$ and $V_h$ are defined on the same mesh $\mathcal{T}_H$
but with different order of finite element methods. Here, $V_H$
is chosen as the linear finite element space and the fine mesh $V_h$ is $4$-th order finite element space
defined on the mesh $\mathcal T_H$.

In order to validate the convergence results stated in (\ref{Test_1_1})-(\ref{Test_2_0}),
we check the numerical errors corresponding to the linear finite element space $V_H$
with different sizes $H$.  The aim here is to check the dependence of the convergence rate
on the mesh size $H$. The coarse mesh $\mathcal T_H$ is set to be the regular type of uniform mesh.
Figure \ref{Result_Low_Order} shows the corresponding convergence behaviors for the first eigenfunction
by Algorithm \ref{Algorithm_k} (or Algorithm \ref{Algorithm_1}) with the coarse space being the linear finite
element space on the mesh with size $H=\sqrt{2}/8$, $\sqrt{2}/16$, $\sqrt{2}/32$ and $\sqrt{2}/64$.
We can find the corresponding convergence rate are $0.044633$, $0.012493$,
$0.0032218$ and $0.00081231$. These results show that the augmented subspace method defined by
Algorithms \ref{Algorithm_k} and \ref{Algorithm_1} should have second order convergence which
validates the results (\ref{Test_1_1})-(\ref{Test_2_0}).
\begin{figure}[http!]
\centering
\includegraphics[width=6cm,height=4.5cm]{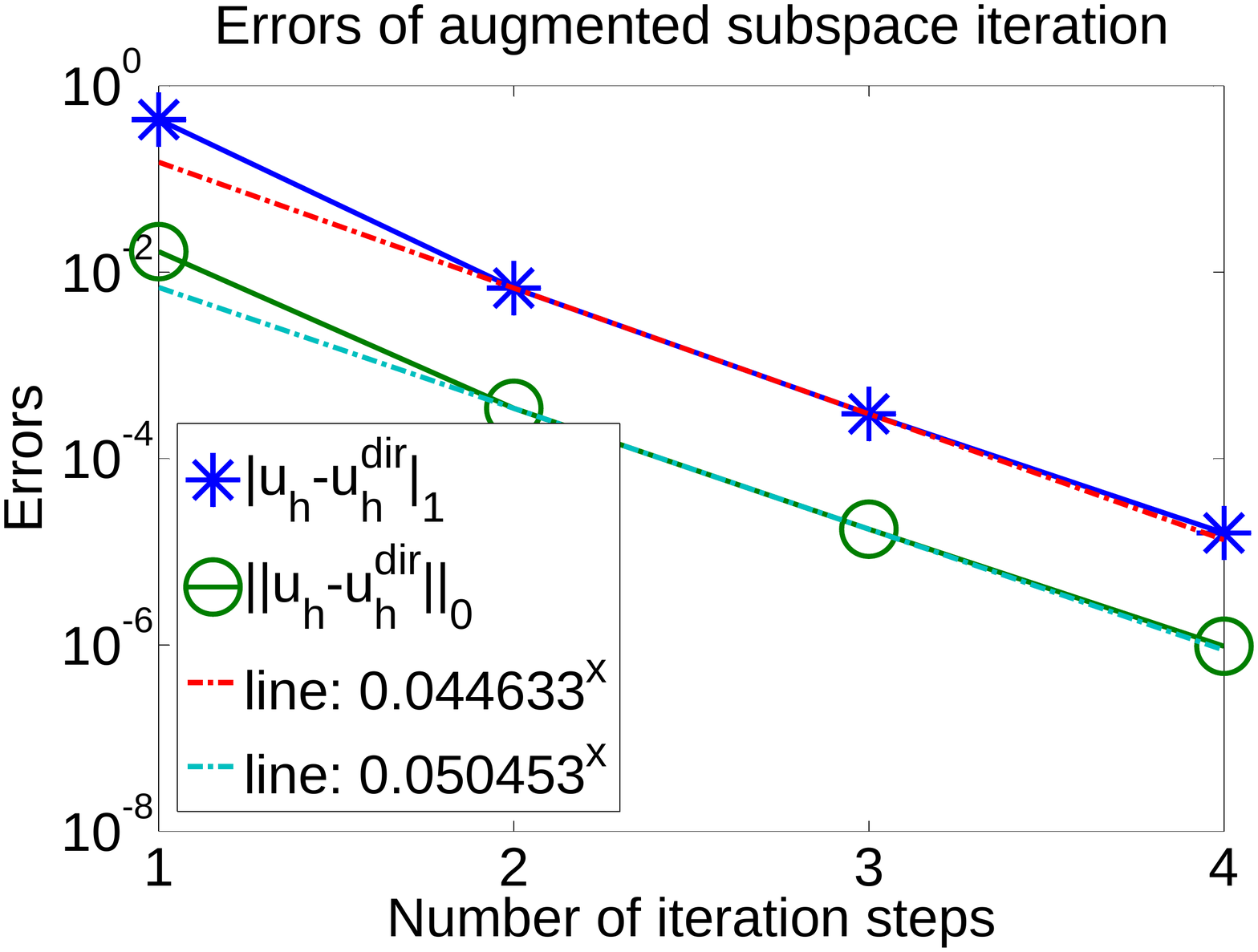}
\includegraphics[width=6cm,height=4.5cm]{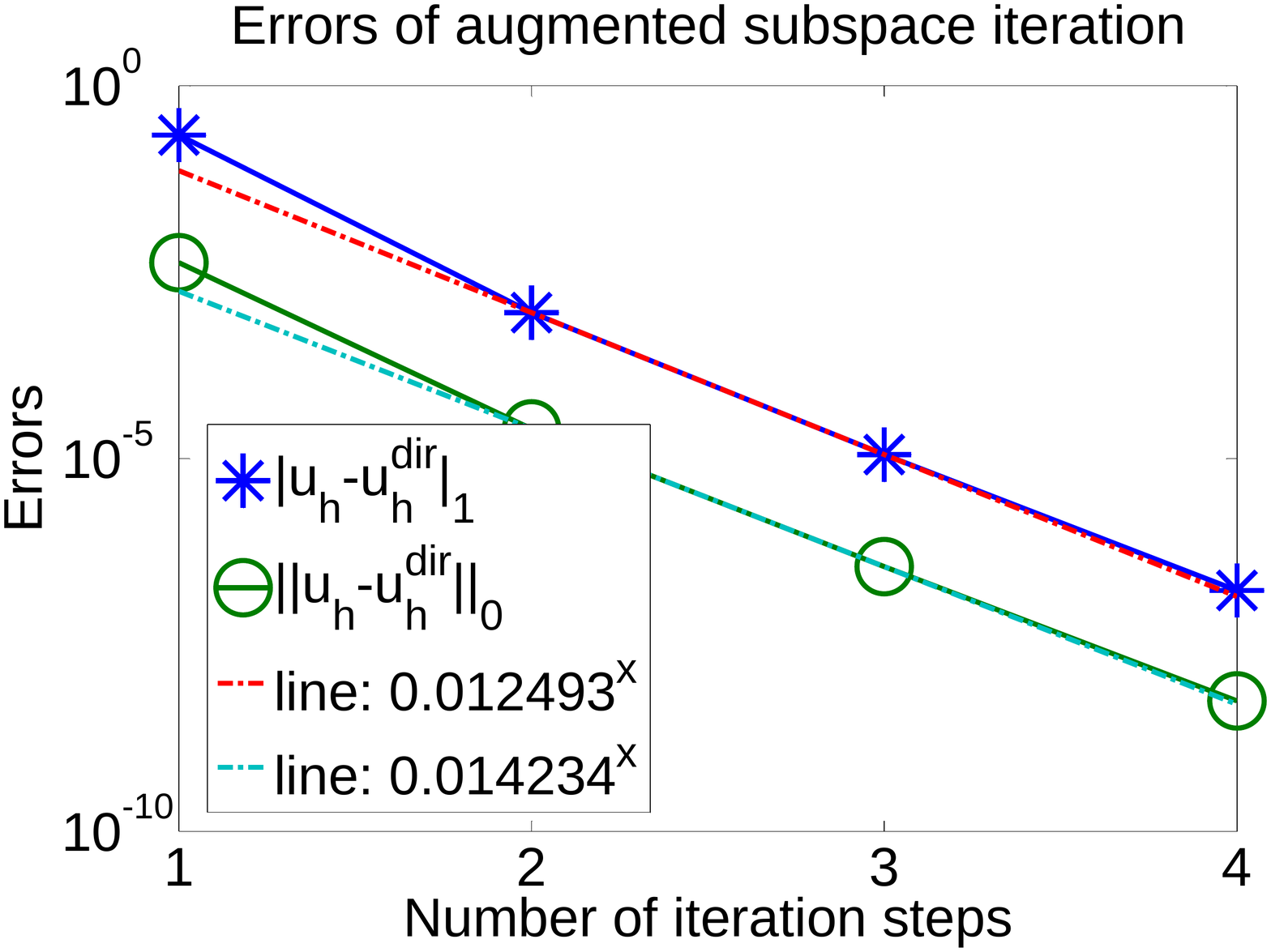}\\
\includegraphics[width=6cm,height=4.5cm]{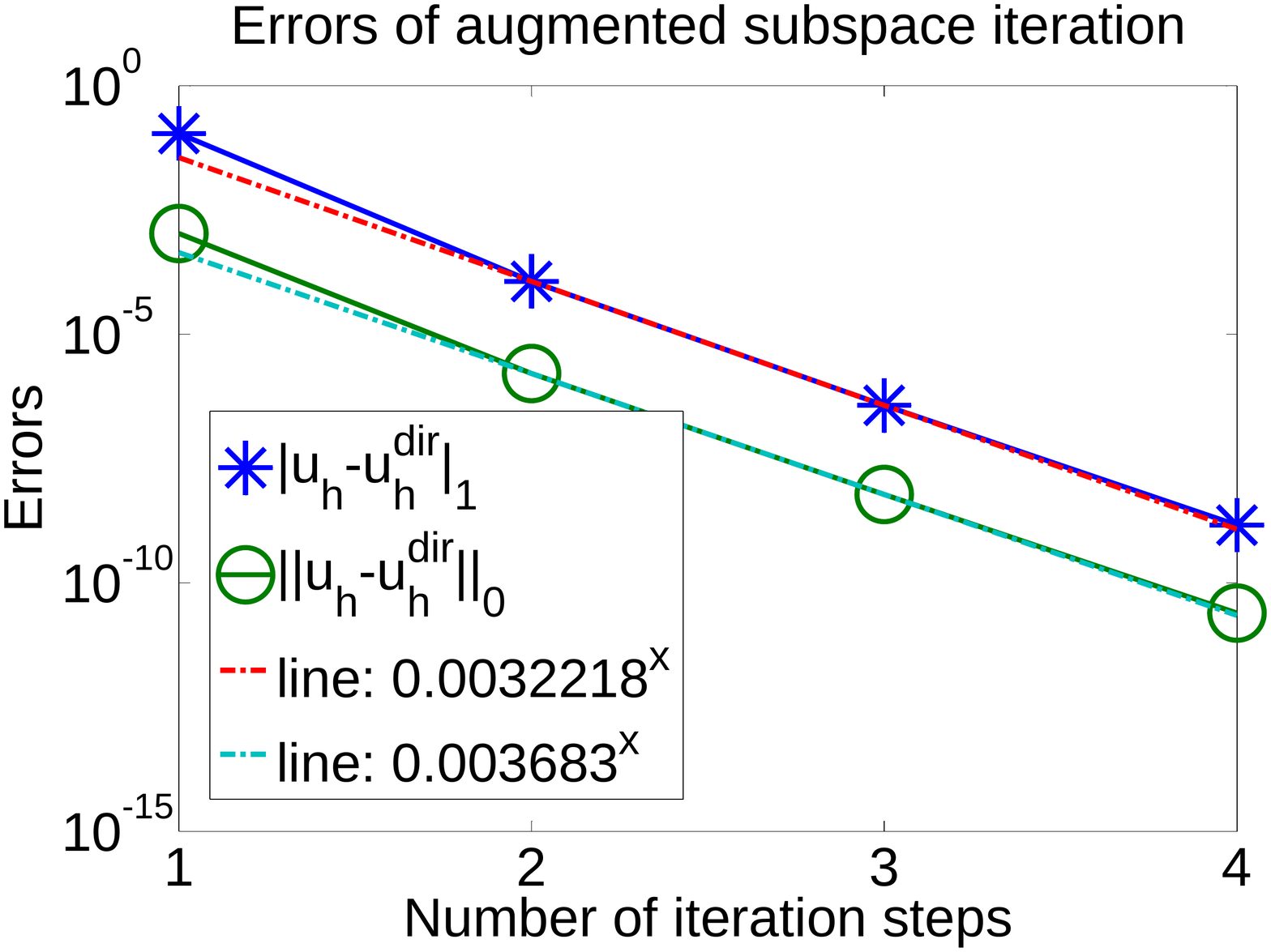}
\includegraphics[width=6cm,height=4.5cm]{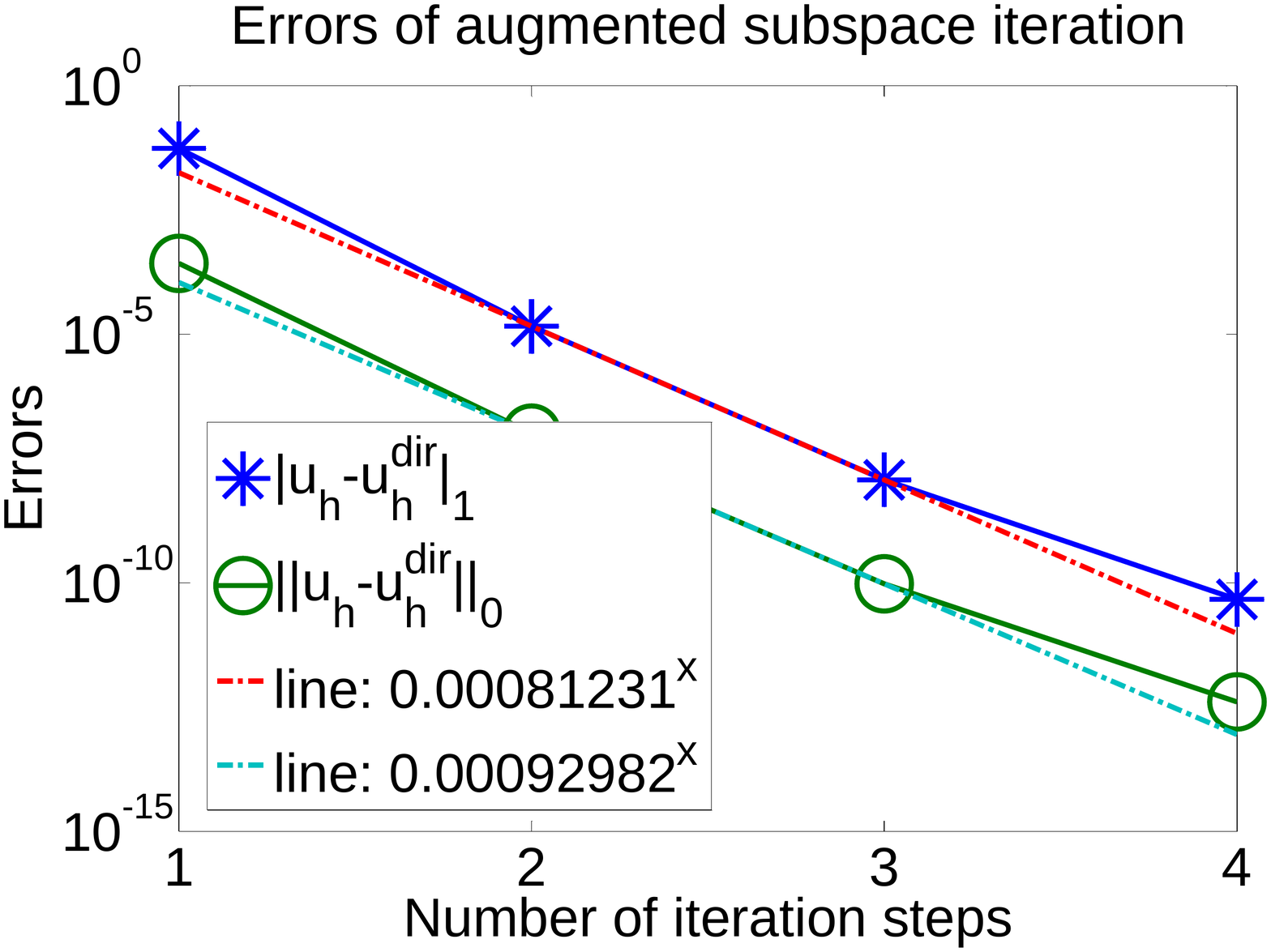}
\caption{The convergence behaviors for the first eigenfunction by Algorithm \ref{Algorithm_k}
with the coarse space being the linear finite element space on the mesh with size $H=\sqrt{2}/8$, $\sqrt{2}/16$, $\sqrt{2}/32$ and $\sqrt{2}/64$.
The corresponding convergence rates are $0.044633$, $0.012493$, $0.0032218$ and $0.00081231$.}\label{Result_Low_Order}
\end{figure}

Here, we also check the performance of Algorithm \ref{Algorithm_k} for computing the smallest $4$ eigenpairs.
Figure \ref{Result_Low_Order_4} shows the corresponding convergence behaviors  for the smallest $4$ eigenfunctions
by Algorithm \ref{Algorithm_k} with the coarse space being the linear finite element space on the mesh with size $H=\sqrt{2}/8$, $\sqrt{2}/16$, $\sqrt{2}/32$ and $\sqrt{2}/64$.
We can find the corresponding convergence rate are $0.35452$, $0.12177$, $0.032864$ and $0.007999$.
Furthermore, from Figures \ref{Result_Low_Order} and \ref{Result_Low_Order_4}, we can find
the convergence rate for the $4$-th eigenfucntion is slower than that for the $1$-st eigenfunction which
is consistent with Theorem \ref{Algorithm_k}.
\begin{figure}[http!]
\centering
\includegraphics[width=6cm,height=4.5cm]{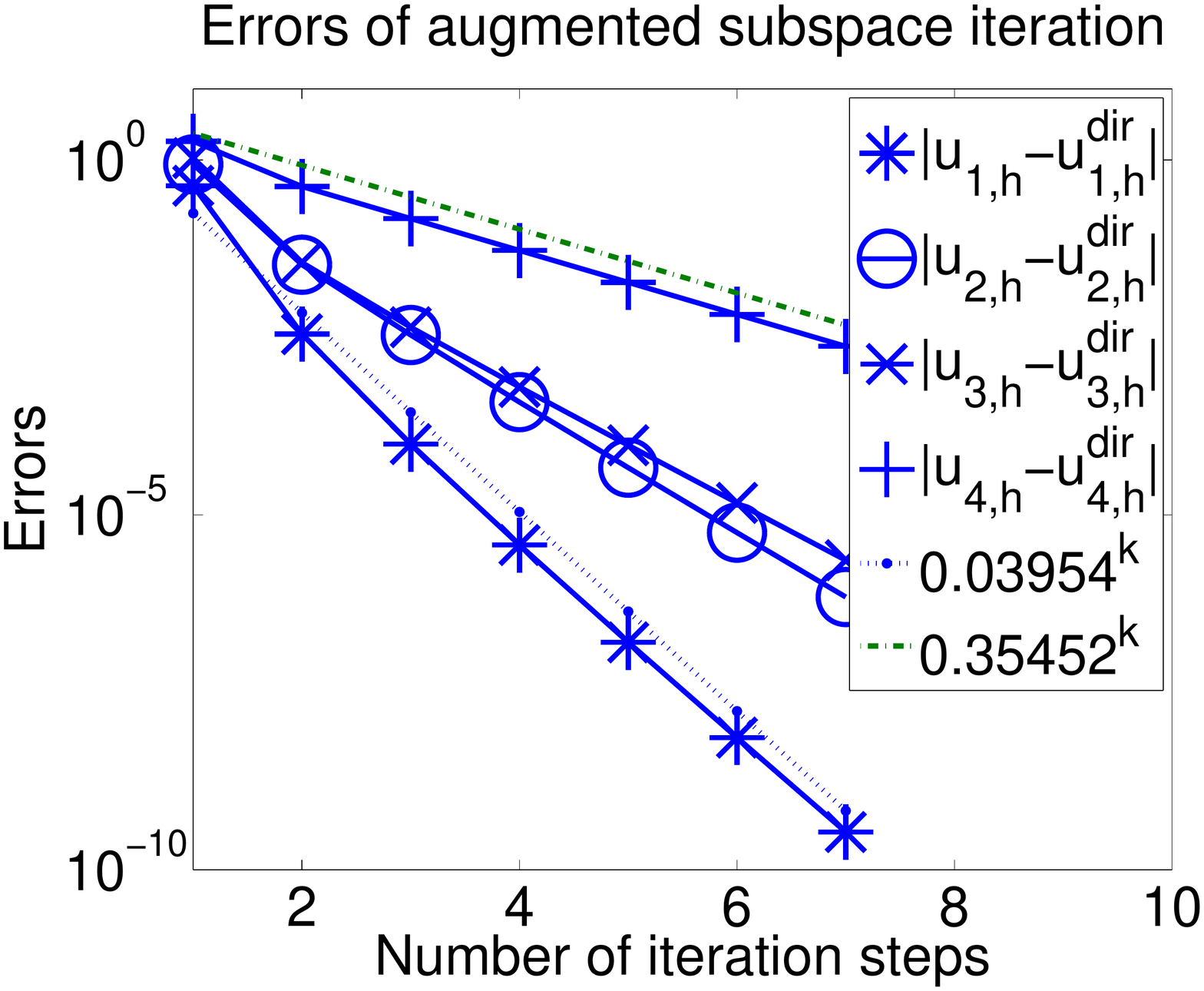}
\includegraphics[width=6cm,height=4.5cm]{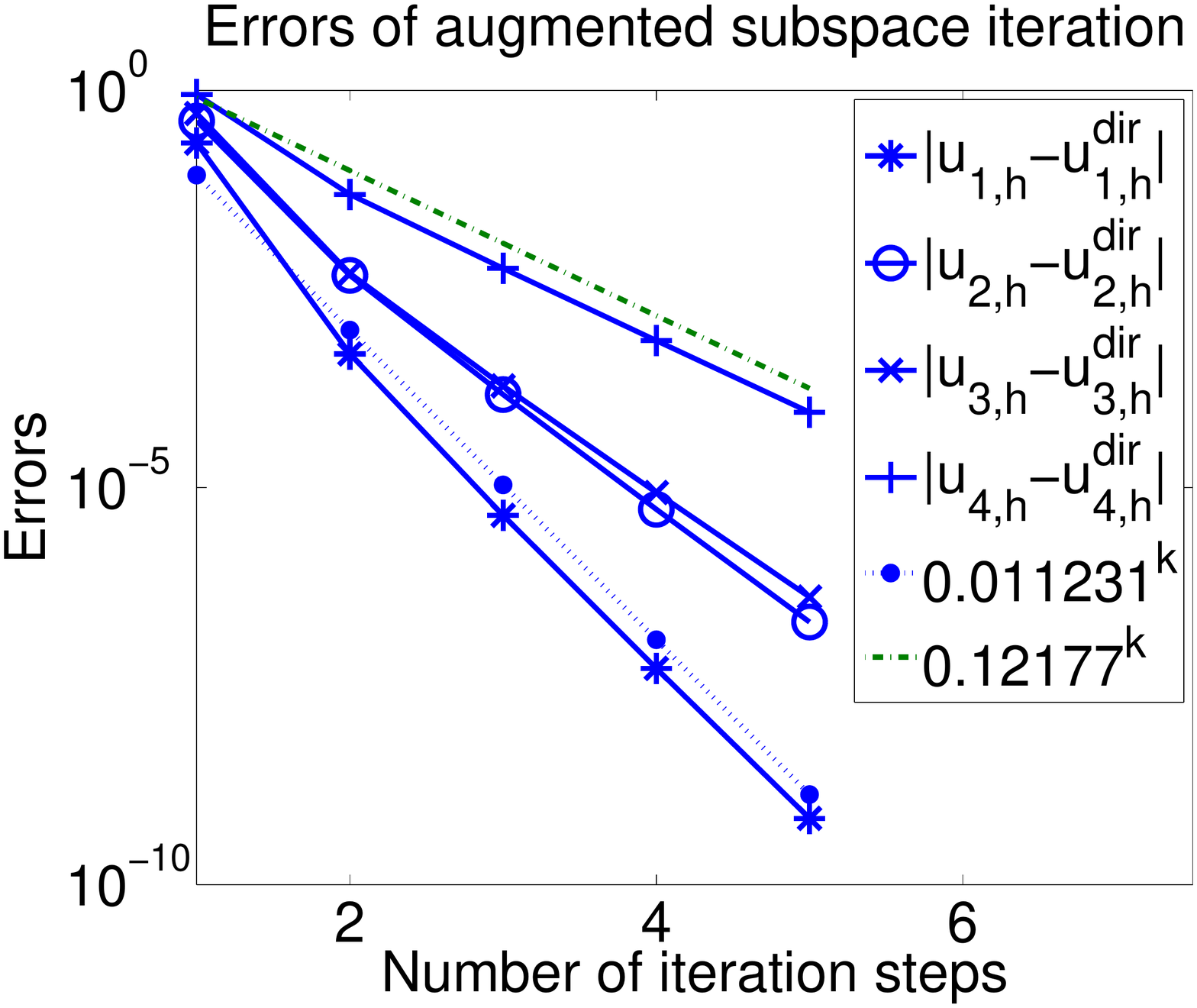}\\
\includegraphics[width=6cm,height=4.5cm]{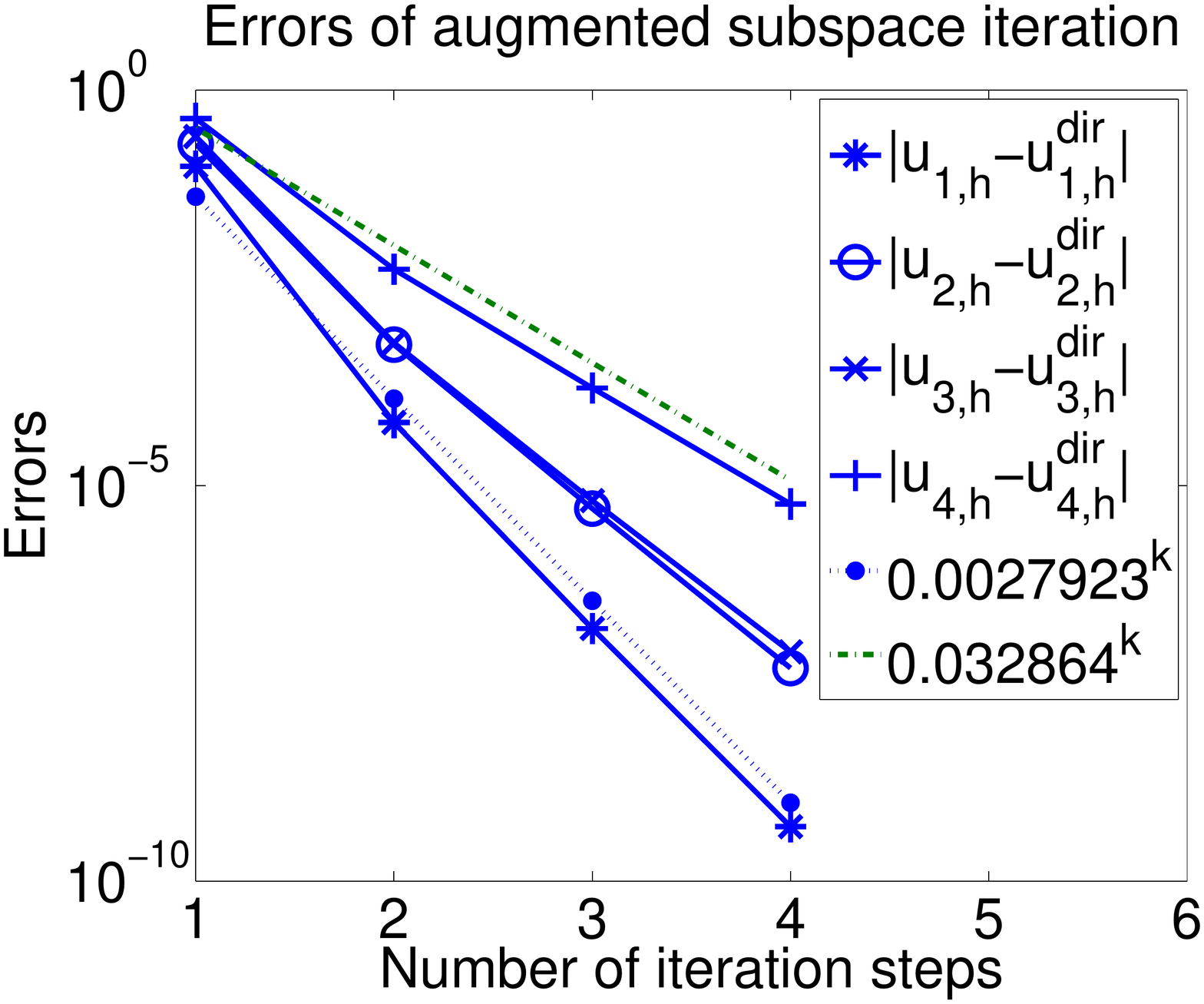}
\includegraphics[width=6cm,height=4.5cm]{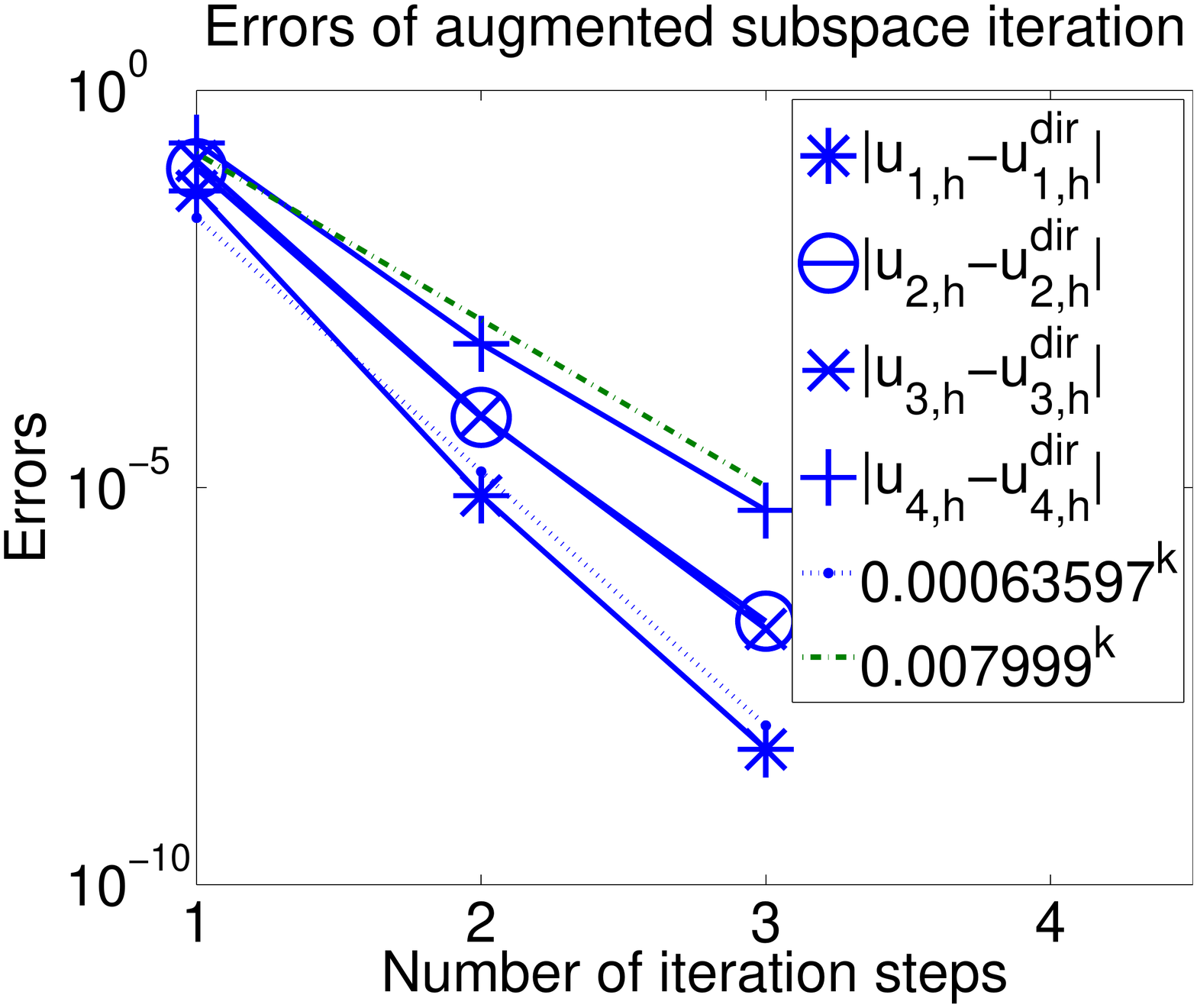}
\caption{The convergence behaviors for the smallest $4$ eigenfunction by Algorithm \ref{Algorithm_k}
with the coarse space being the linear finite element space on the mesh with size $H=\sqrt{2}/8$, $\sqrt{2}/16$, $\sqrt{2}/32$ and $\sqrt{2}/64$.
The corresponding convergence rates are $0.35452$, $0.12177$, $0.032864$ and $0.007999$.}\label{Result_Low_Order_4}
\end{figure}

The next task is to check the performance of Algorithm \ref{Algorithm_1} for computing the only $4$-th eigenpair.
Figure \ref{Result_Low_Order_4_Only} shows the corresponding convergence behaviors  for the only $4$-th eigenfunctions
by Algorithm \ref{Algorithm_1} with the coarse space being the linear finite element space on
the mesh with size $H=\sqrt{2}/8$, $\sqrt{2}/16$, $\sqrt{2}/32$ and $\sqrt{2}/64$. The corresponding convergence rate shown in
Figure \ref{Result_Low_Order_4_Only} are $0.35918$, $0.12588$, $0.035169$ and $0.0090917$.
These results show that the augmented subspace method defined by
Algorithm \ref{Algorithm_1} has second order convergence which validate the results (\ref{Test_2_1})-(\ref{Test_2_0}).
\begin{figure}[http!]
\centering
\includegraphics[width=6cm,height=4.5cm]{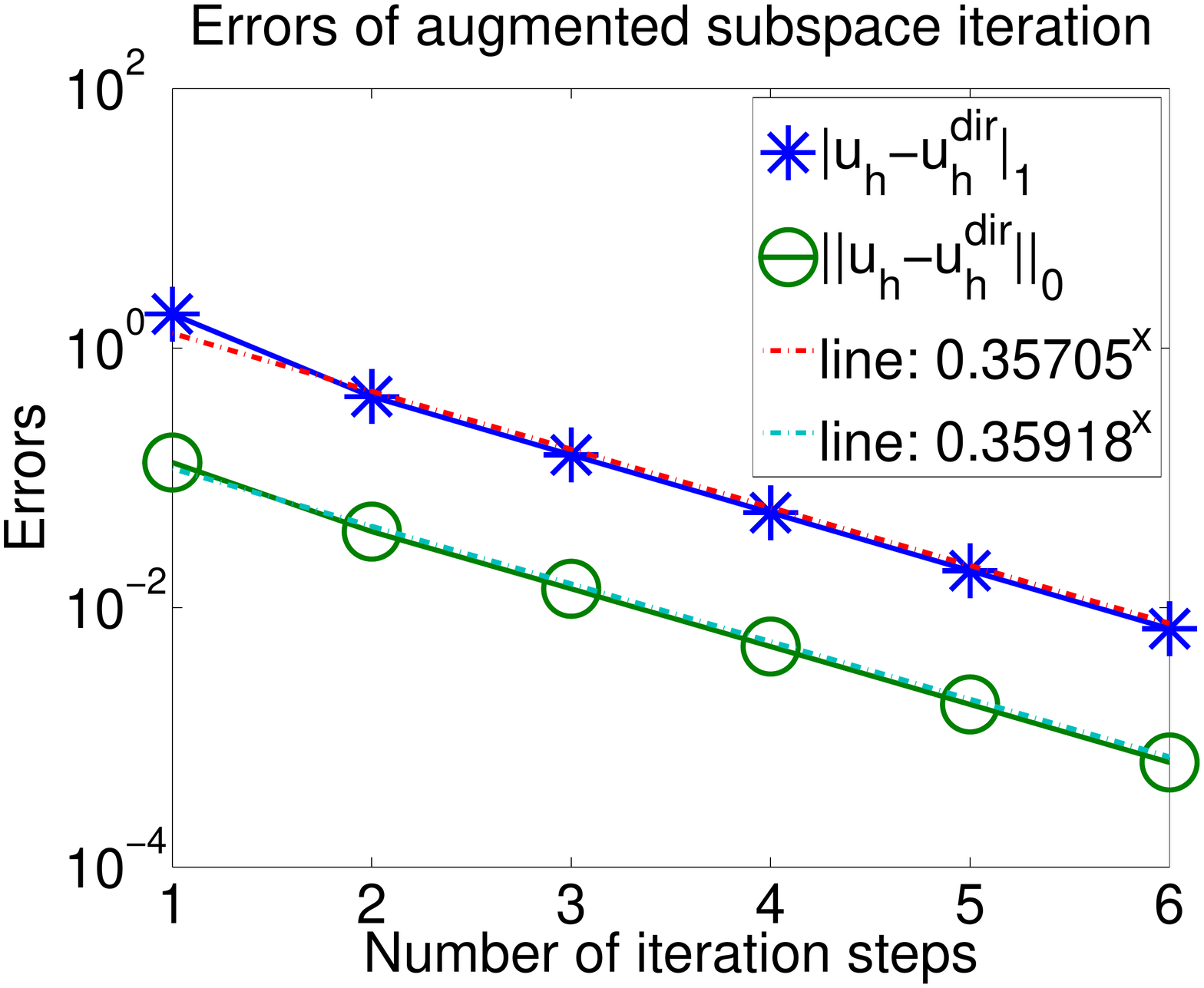}
\includegraphics[width=6cm,height=4.5cm]{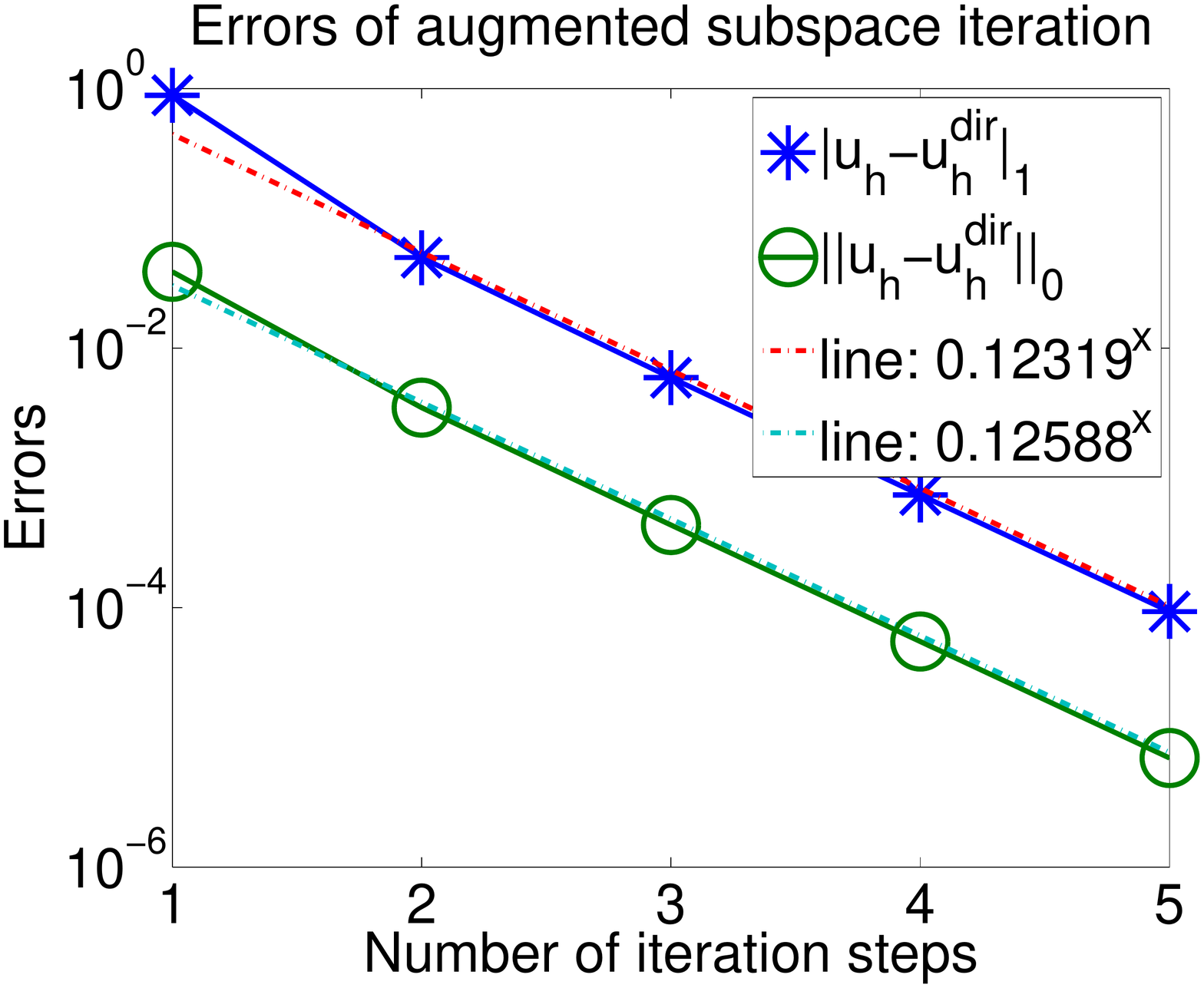}\\
\includegraphics[width=6cm,height=4.5cm]{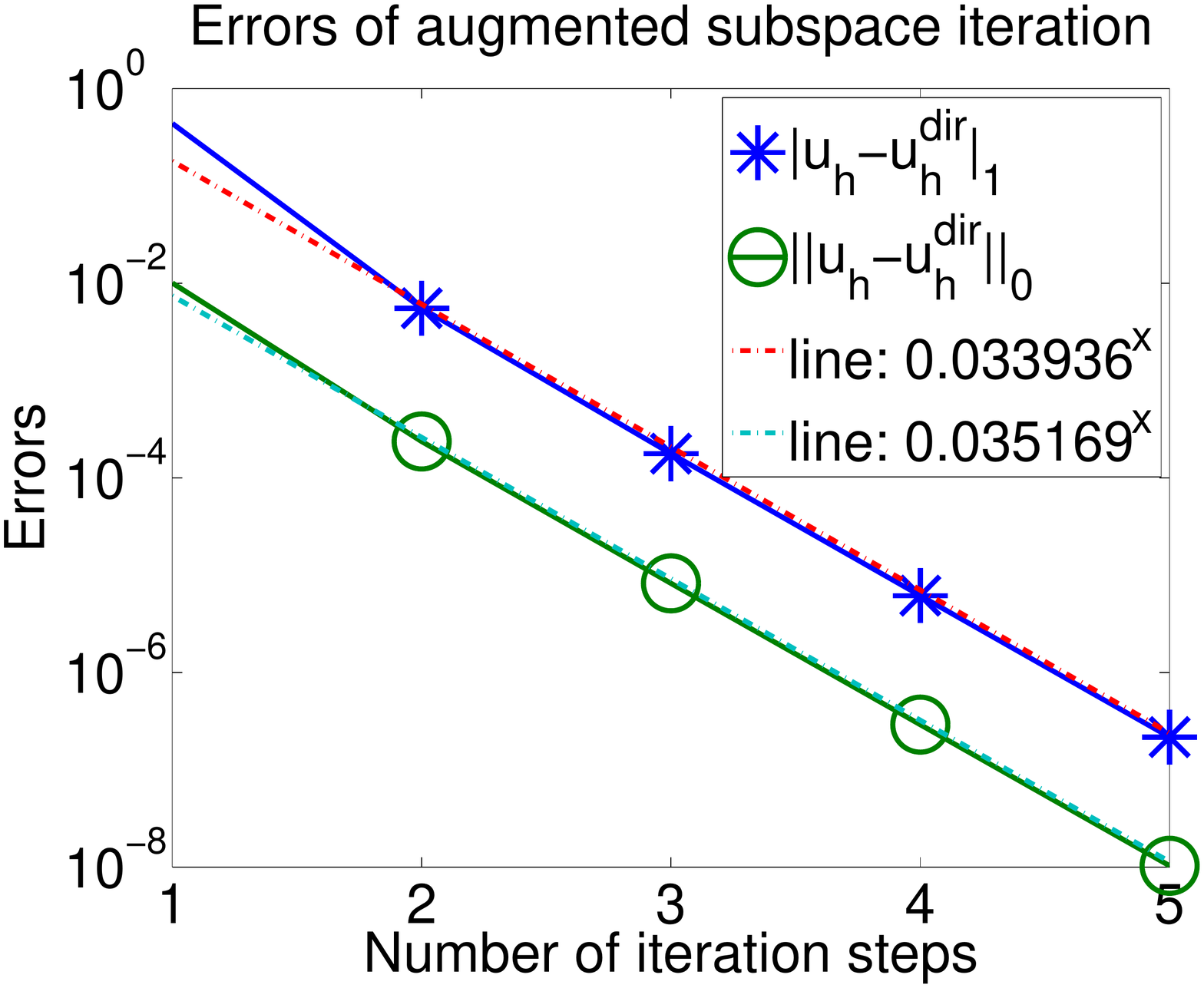}
\includegraphics[width=6cm,height=4.5cm]{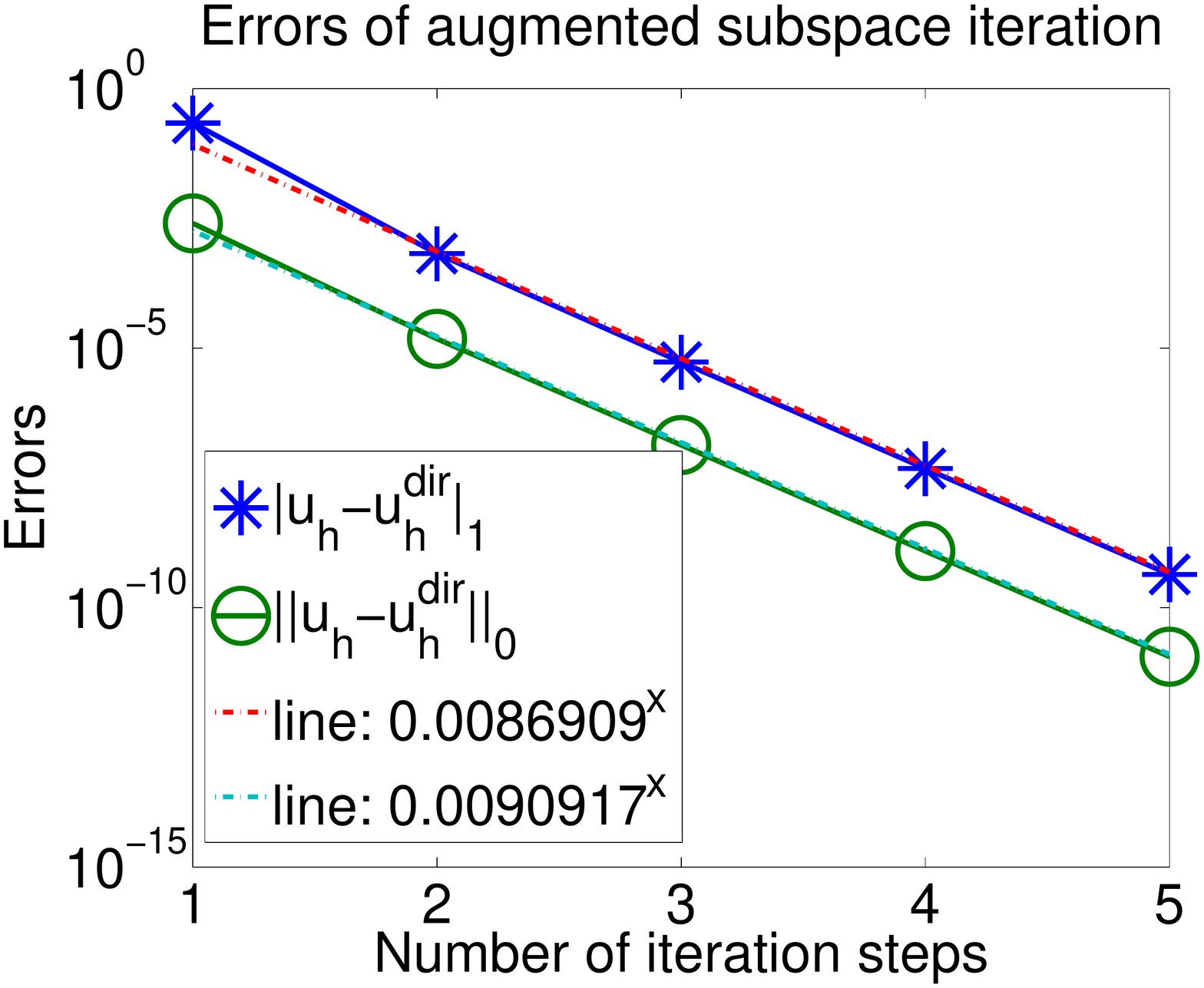}
\caption{The convergence behaviors for the only $4$-th eigenfunction by Algorithm \ref{Algorithm_1}
with the coarse space being the linear finite element space on the mesh with size $H=\sqrt{2}/8$, $\sqrt{2}/16$, $\sqrt{2}/32$ and $\sqrt{2}/64$.
The corresponding convergence rates are $0.35918$, $0.12588$, $0.035169$ and $0.0090917$.}\label{Result_Low_Order_4_Only}
\end{figure}

\subsection{Augmented subspace by the finite element space on the coarse mesh}
In the second subsection, $V_h$ is chosen as the linear finite element space defined on the finer mesh $\mathcal T_h$.
For this aim, we start from the coarse mesh $\mathcal T_H$ to produce the finer mesh
by the regular refinement. In the numerical tests here, we set the size $h=1/256$ for the finer mesh $\mathcal T_h$.
Here, $V_H$ is  chosen as the linear finite element space defined on the coarse mesh $\mathcal T_H$.
The initial eigenfunction approximation is also produced by solving the eigenvalue
problems on the coarse space $V_H$. Then we do the iteration steps by the augmented subspace
method defined by Algorithms \ref{Algorithm_k} and \ref{Algorithm_1}.

In order to validate the convergence results stated in (\ref{Test_1_1})-(\ref{Test_2_0}),
we  also check the numerical errors corresponding to the linear finite element space $V_H$
with different sizes $H$.  The aim is to check the dependence of the convergence rate
on the mesh size $H$. Here, the coarse mesh $\mathcal T_H$ is also set to be the regular type of uniform mesh.

Figure \ref{Result_Coarse_Mesh} shows the convergence behaviors for the first eigenfunction by
the augmented subspace methods corresponding to the coarse mesh sizes $H=\sqrt{2}/4$, $\sqrt{2}/8$, $\sqrt{2}/16$ and $\sqrt{2}/32$.
The corresponding convergence rates are $0.13142$, $0.048523$, $0.013652$ and $0.0035056$.
These results show that the augmented subspace method defined by Algorithms \ref{Algorithm_k} and \ref{Algorithm_1}
should have second order convergence
which also validates the results (\ref{Test_1_1})-(\ref{Test_1_0}).
\begin{figure}[http!]
\centering
\includegraphics[width=6cm,height=4.5cm]{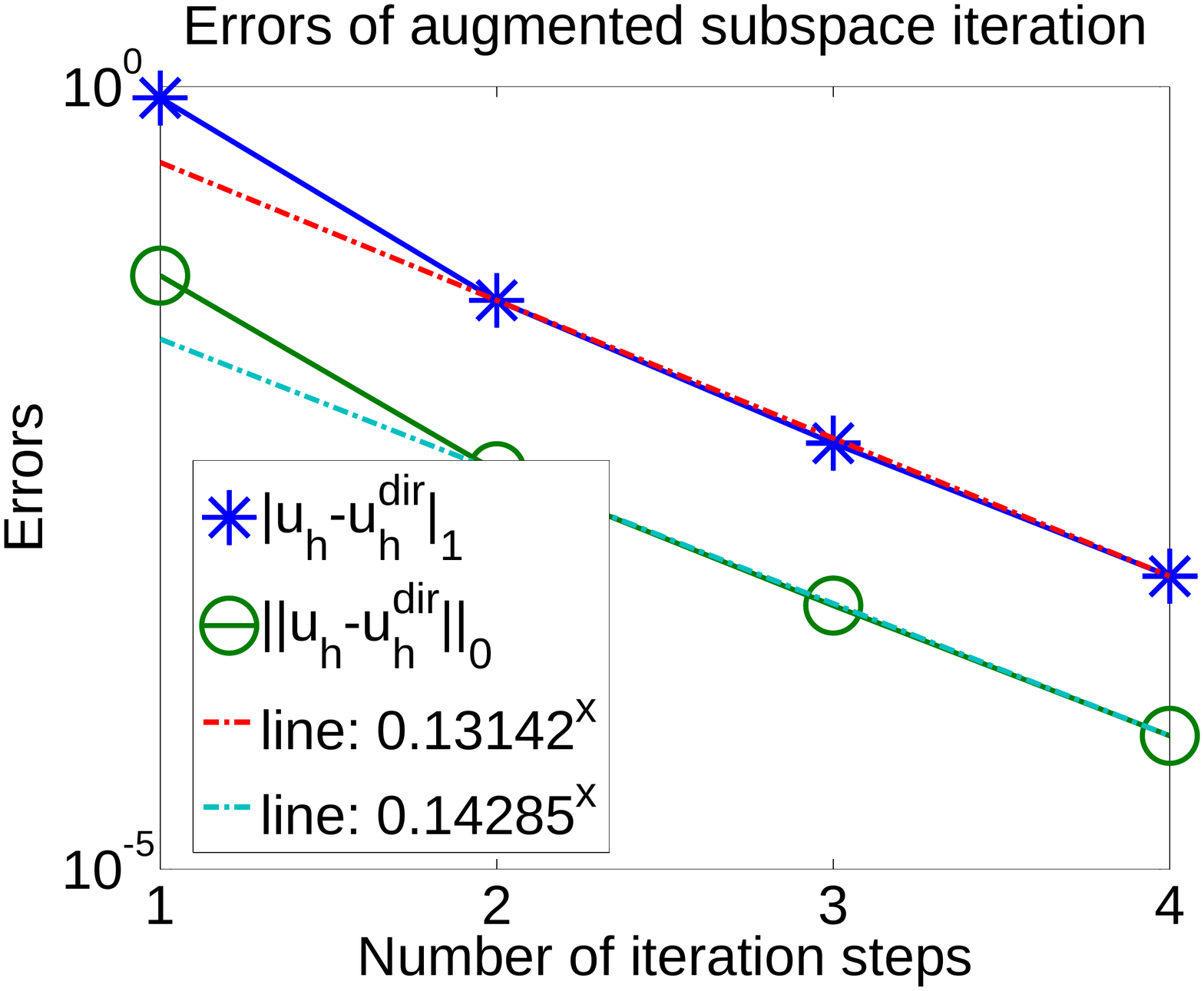}
\includegraphics[width=6cm,height=4.5cm]{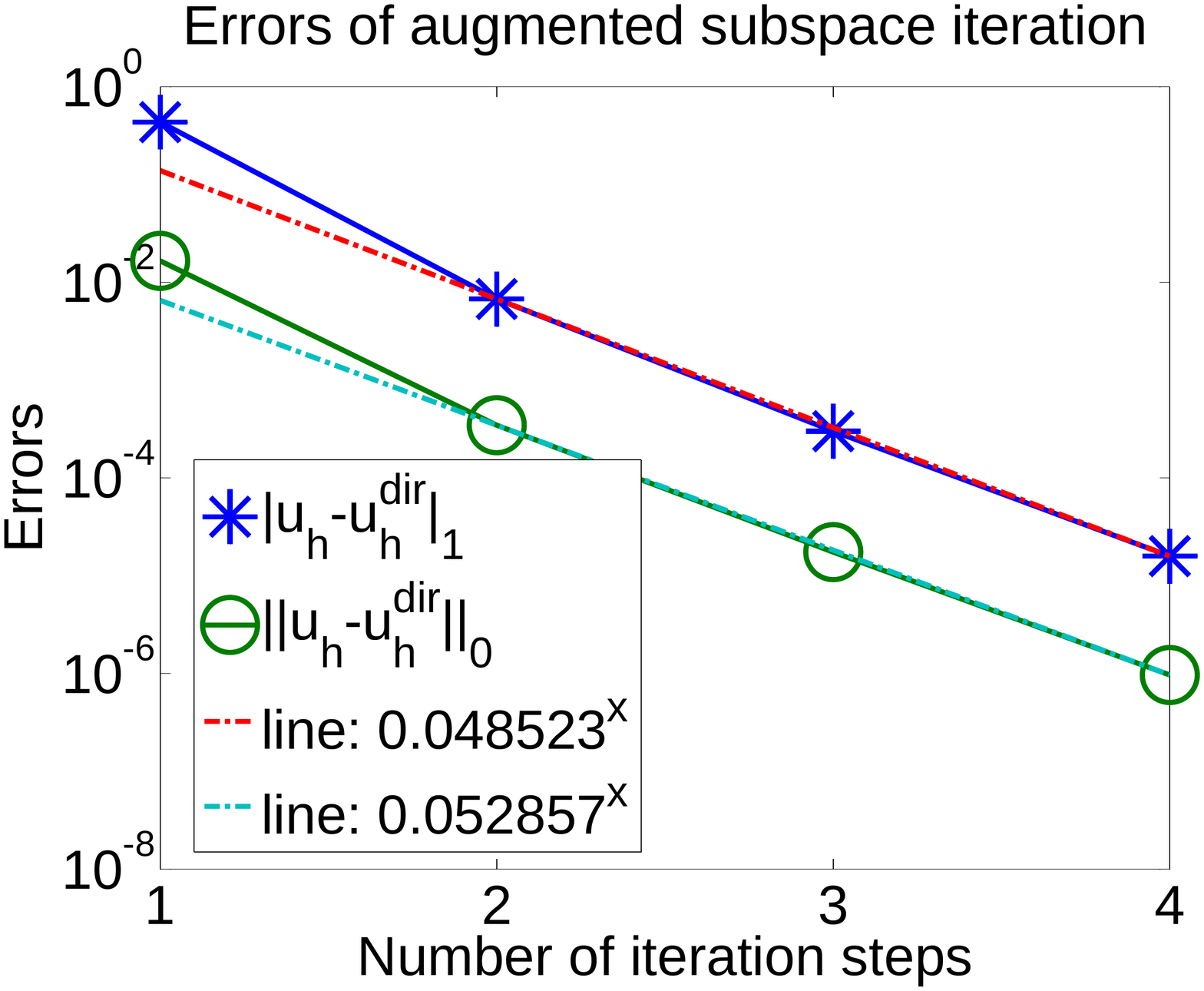}\\
\includegraphics[width=6cm,height=4.5cm]{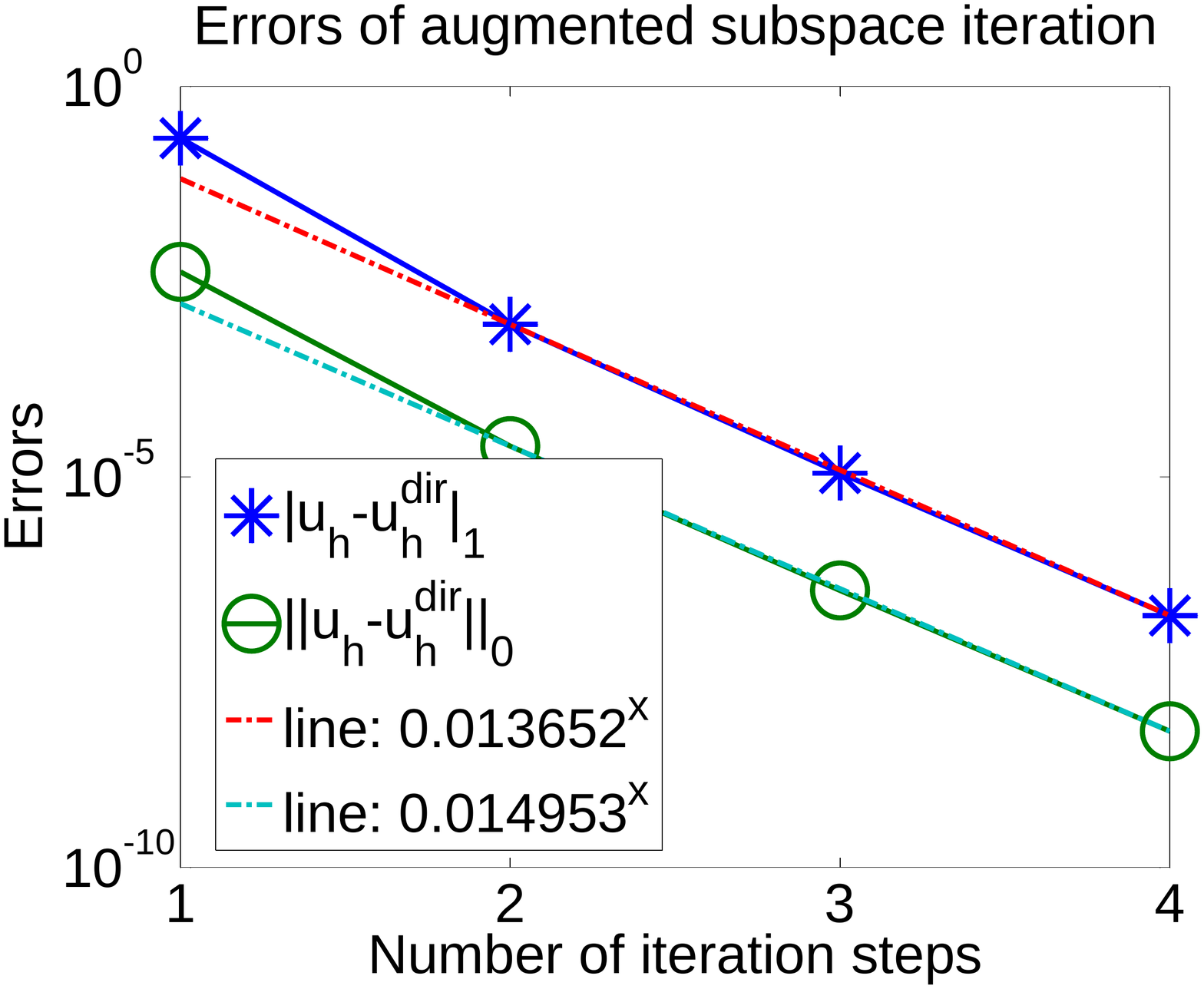}
\includegraphics[width=6cm,height=4.5cm]{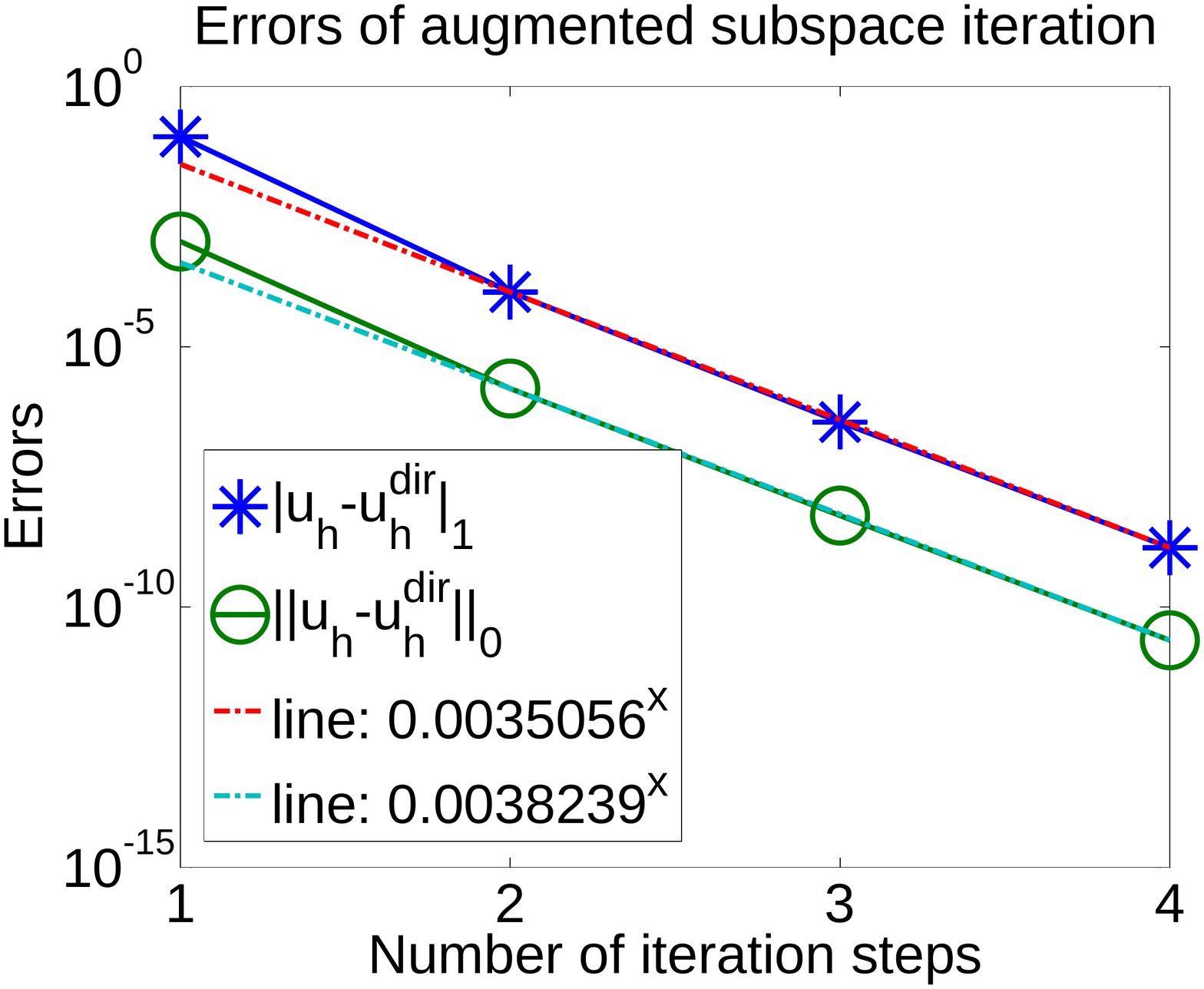}
\caption{The convergence behaviors for the first eigenfunction by Algorithm \ref{Algorithm_k}
corresponding to the coarse mesh size $H=\sqrt{2}/4$, $\sqrt{2}/8$, $\sqrt{2}/16$ and $\sqrt{2}/32$.
The corresponding convergence rates are $0.13142$, $0.048523$, $0.013652$ and  $0.0035056$.}\label{Result_Coarse_Mesh}
\end{figure}

Then, we check the performance of Algorithm \ref{Algorithm_k} for computing the smallest $4$ eigenpairs.
Figure \ref{Result_Coarse_Mesh_4} shows the corresponding convergence behaviors  for the smallest $4$ eigenfunctions
by Algorithm \ref{Algorithm_k} with the coarse space being the linear finite element space on the mesh with size $H=\sqrt{2}/8$, $\sqrt{2}/16$, $\sqrt{2}/32$ and $\sqrt{2}/64$.
We can find that the corresponding convergence rate are $0.31838$, $0.09979$, $0.026024$ and $0.0068251$.
Furthermore, from Figures \ref{Result_Coarse_Mesh} and \ref{Result_Coarse_Mesh_4}, we can find
the convergence rate for the $4$-th eigenfucntion is slower than that for the $1$-st eigenfunction which
is consistent with Theorem \ref{Algorithm_k}.
\begin{figure}[http!]
\centering
\includegraphics[width=6cm,height=4.5cm]{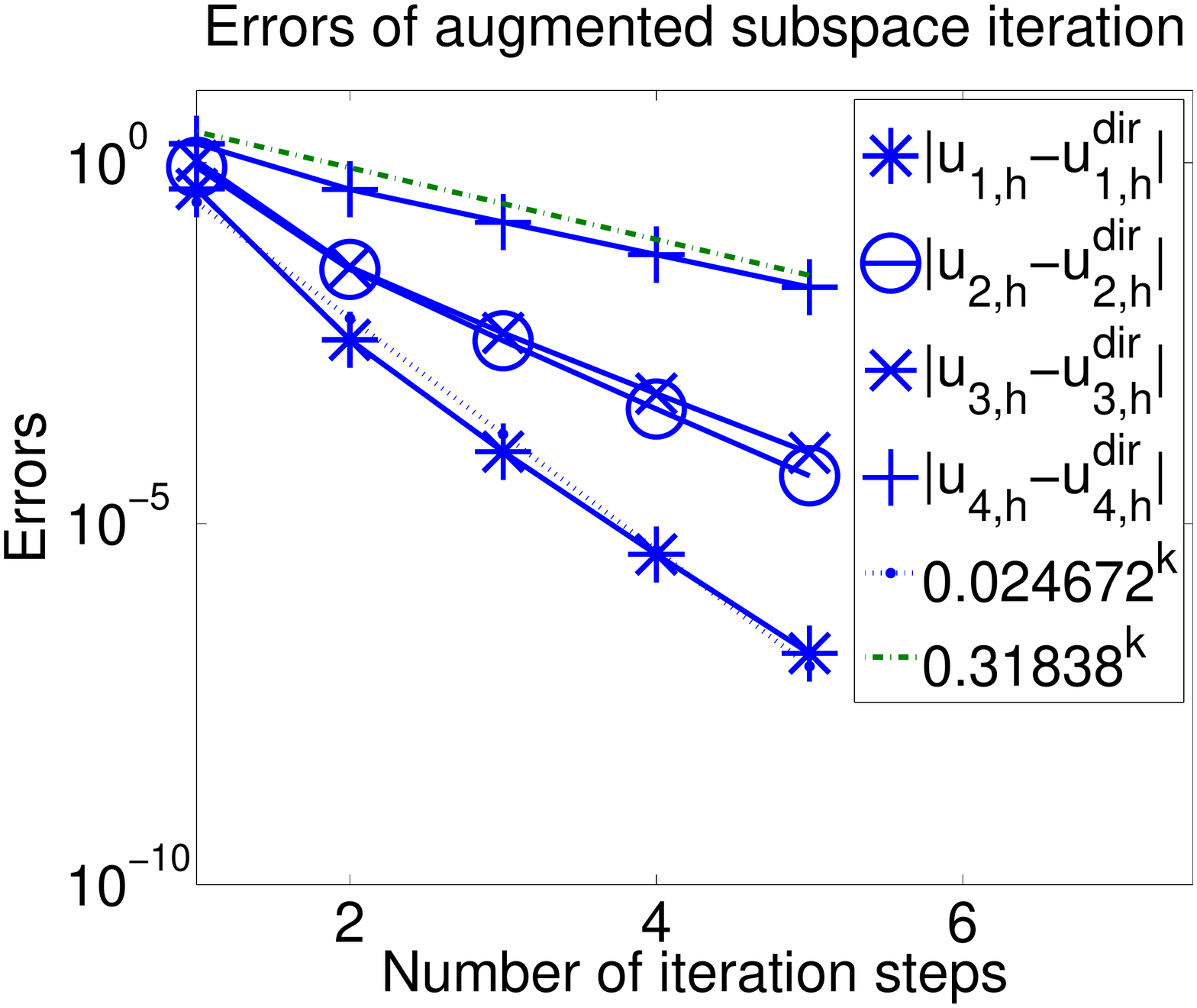}
\includegraphics[width=6cm,height=4.5cm]{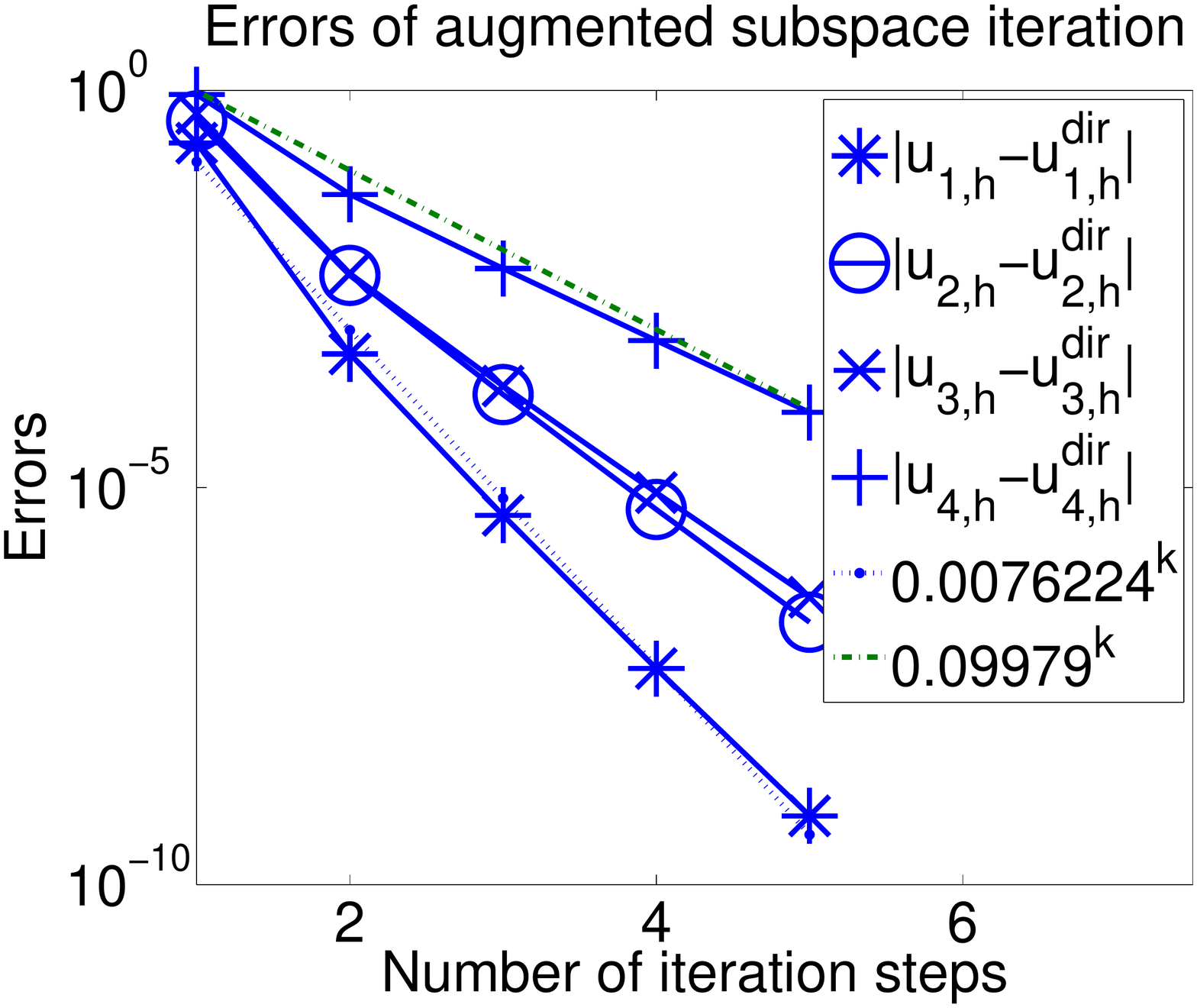}\\
\includegraphics[width=6cm,height=4.5cm]{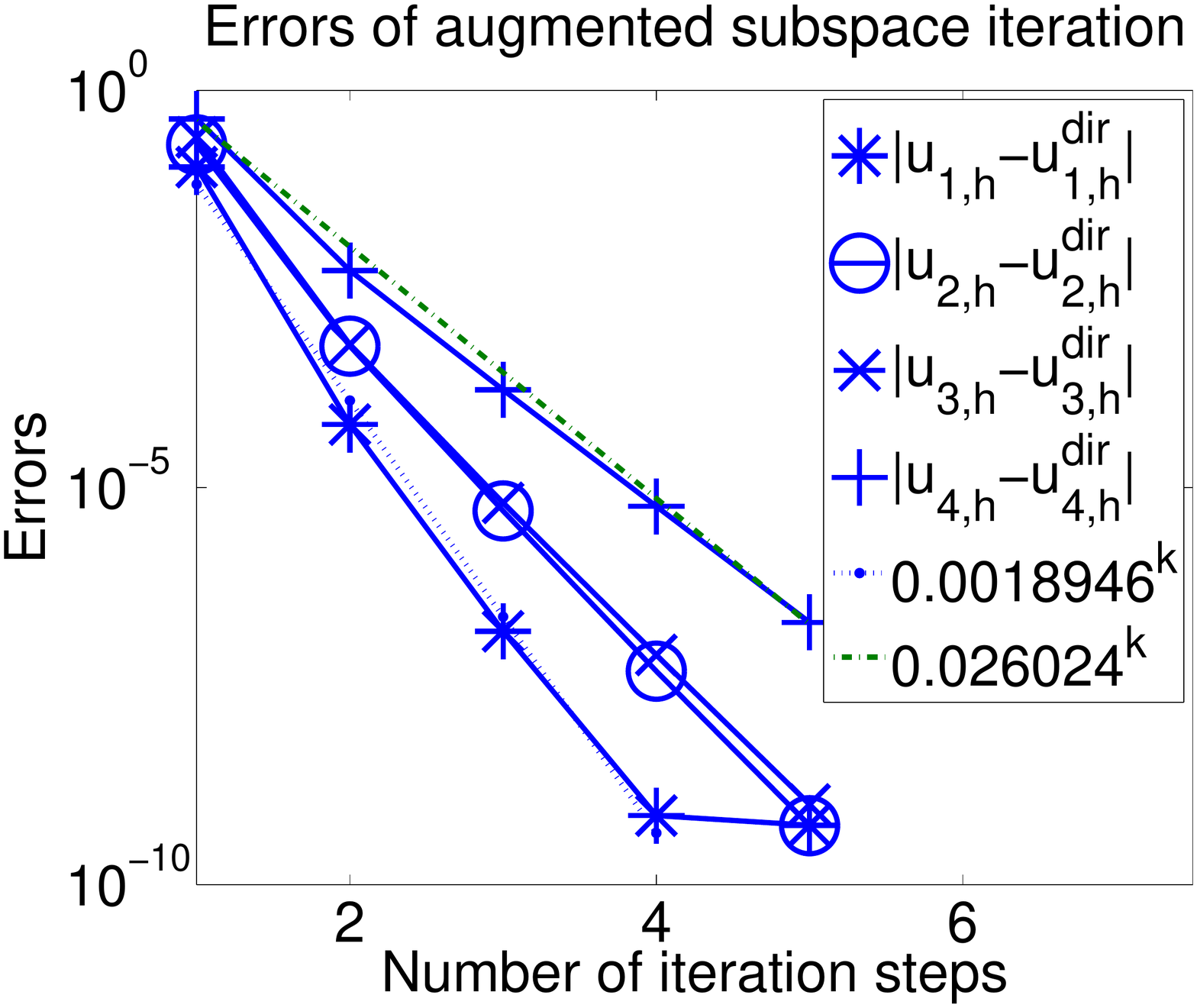}
\includegraphics[width=6cm,height=4.5cm]{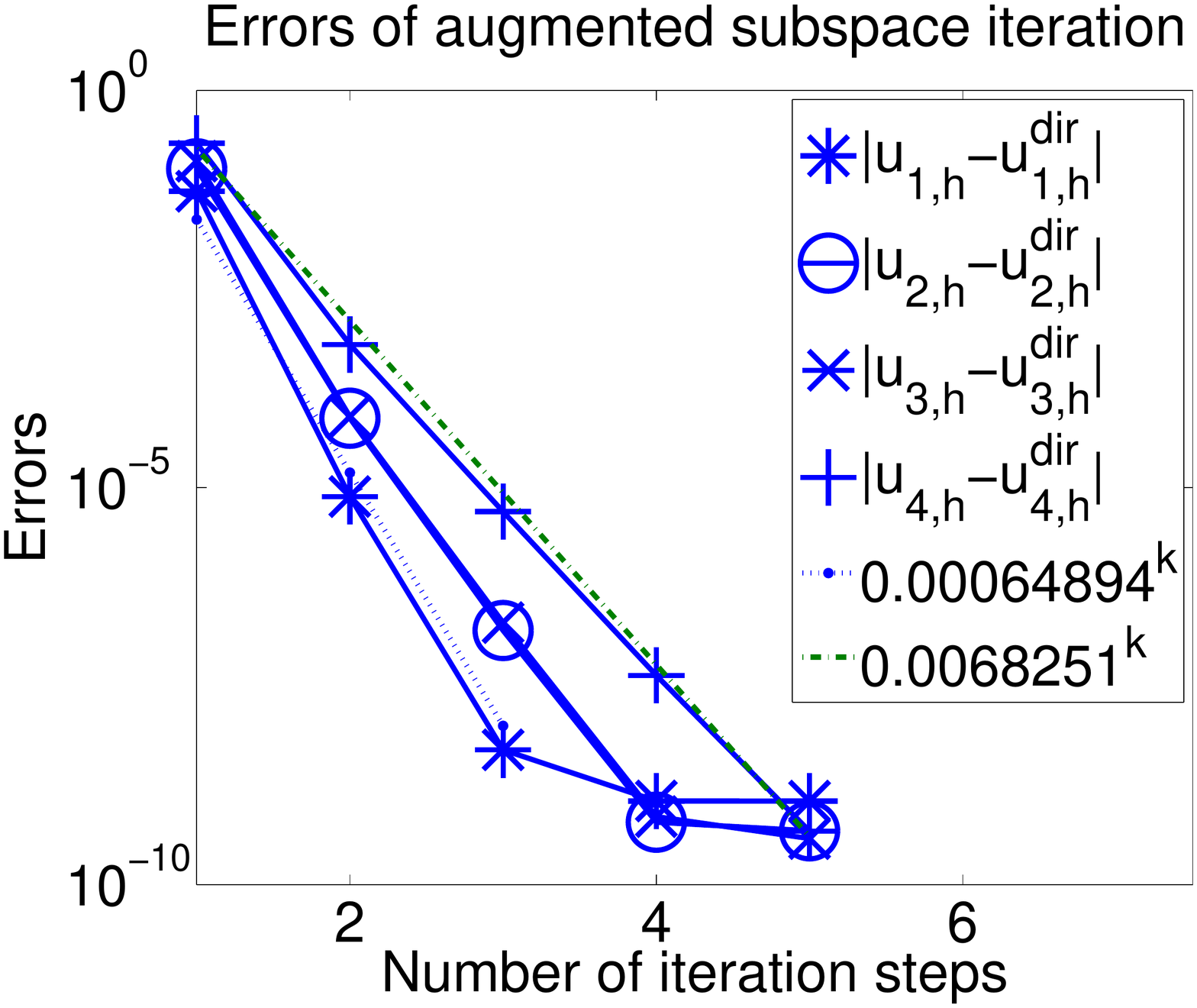}
\caption{The convergence behaviors for the smallest $4$ eigenfunction by Algorithm \ref{Algorithm_k}
with the coarse space being the linear finite element space on the mesh with size $H=\sqrt{2}/8$, $\sqrt{2}/16$, $\sqrt{2}/32$ and $\sqrt{2}/64$.
The corresponding convergence rates are $0.31838$, $0.09979$, $0.026024$ and $0.0068251$.}\label{Result_Coarse_Mesh_4}
\end{figure}

The final task is to check the performance of Algorithm \ref{Algorithm_1} for computing the only $4$-th eigenpair.
Figure \ref{Result_Coarse_Mesh_4_Only} shows the corresponding convergence behaviors  for the only $4$-th eigenfunctions
by Algorithm \ref{Algorithm_1} with the coarse space being the linear finite element space on
the mesh with size $H=\sqrt{2}/8$, $\sqrt{2}/16$, $\sqrt{2}/32$ and $\sqrt{2}/64$. The corresponding convergence rate shown in
Figure \ref{Result_Coarse_Mesh_4_Only} are $0.33687$, $0.11207$, $0.030571$ and $0.0077354$.
These results show that the augmented subspace method defined by
Algorithm \ref{Algorithm_1} has second order convergence which validates the results (\ref{Test_2_1})-(\ref{Test_2_0}).
\begin{figure}[http!]
\centering
\includegraphics[width=6cm,height=4.5cm]{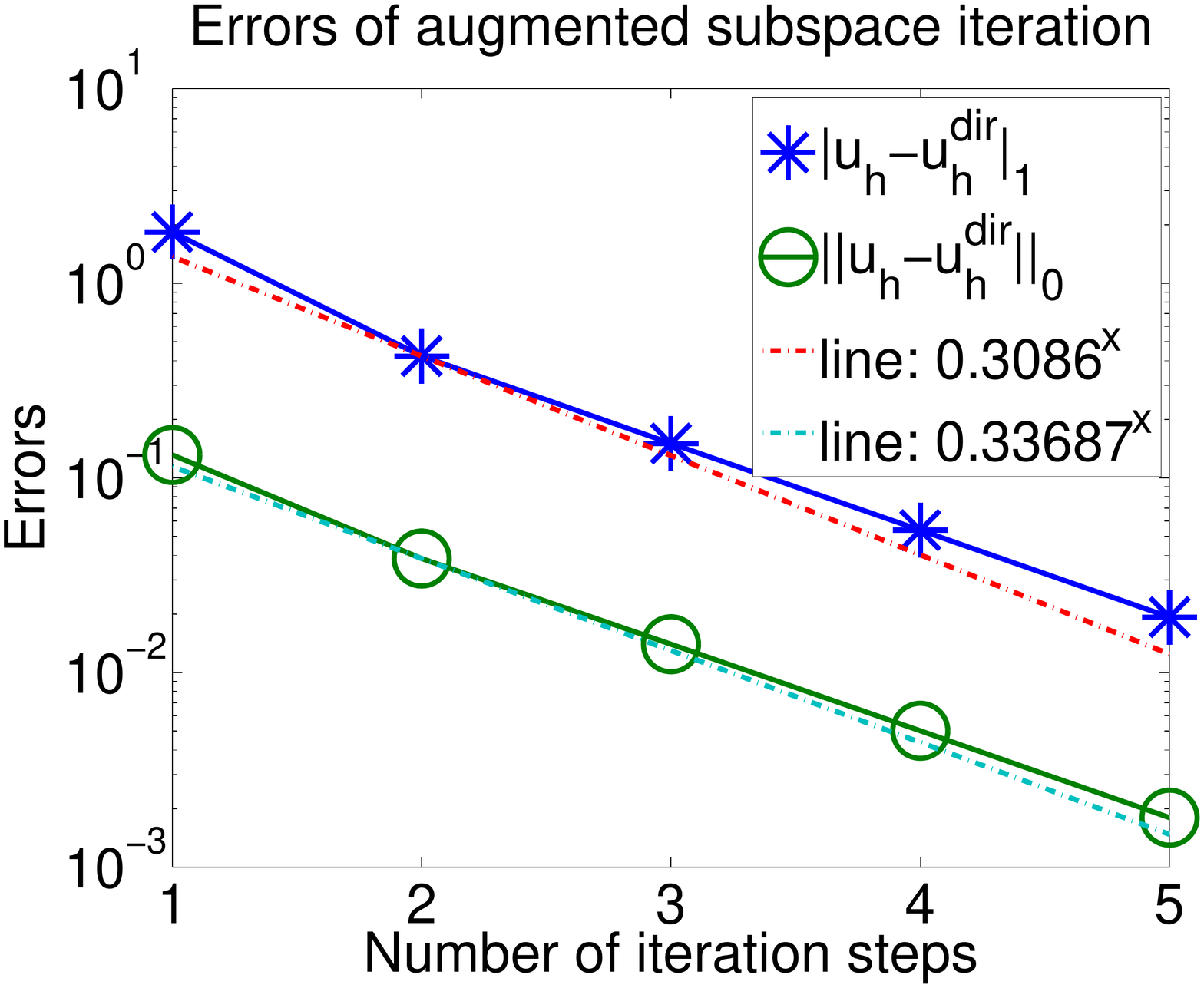}
\includegraphics[width=6cm,height=4.5cm]{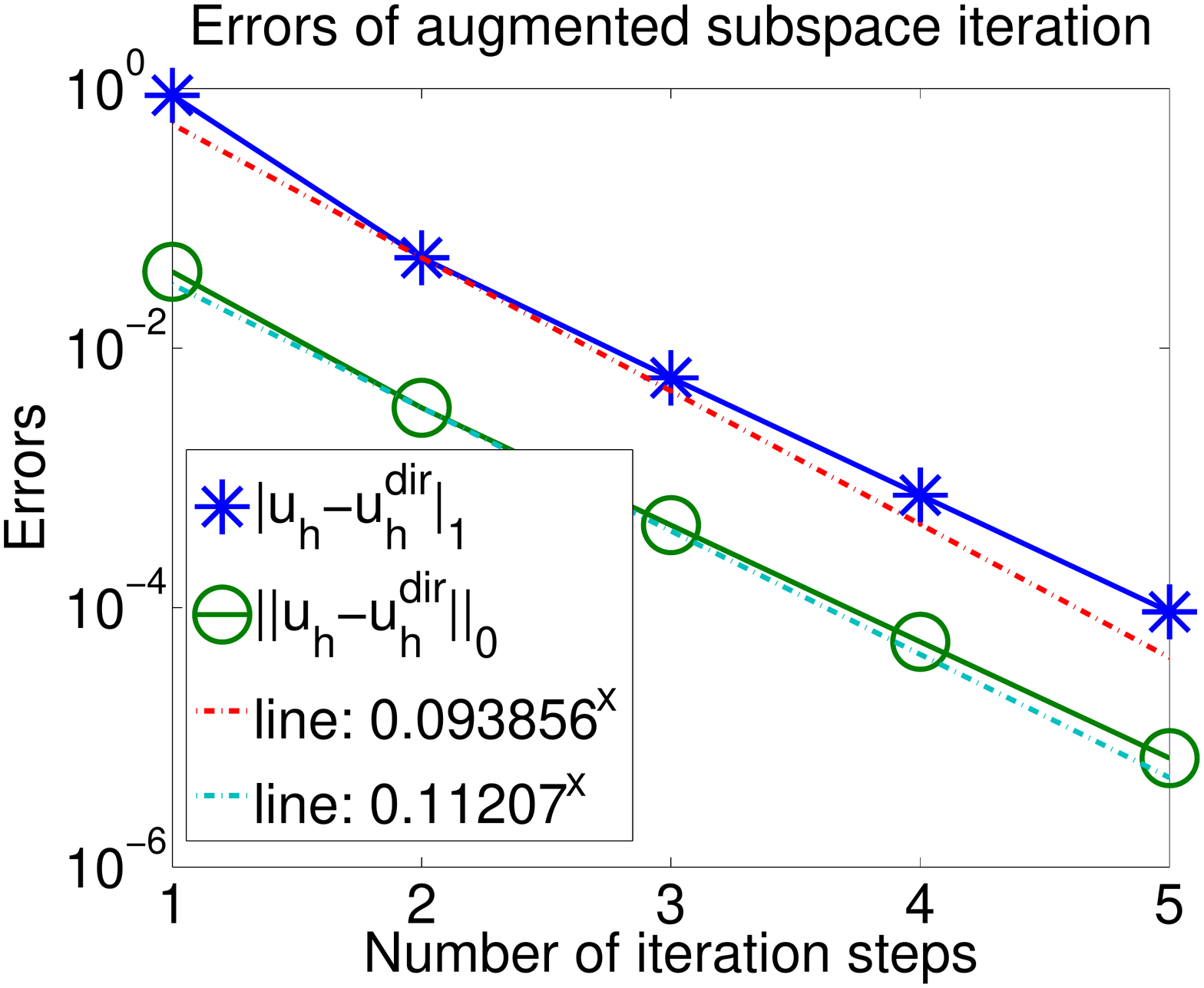}\\
\includegraphics[width=6cm,height=4.5cm]{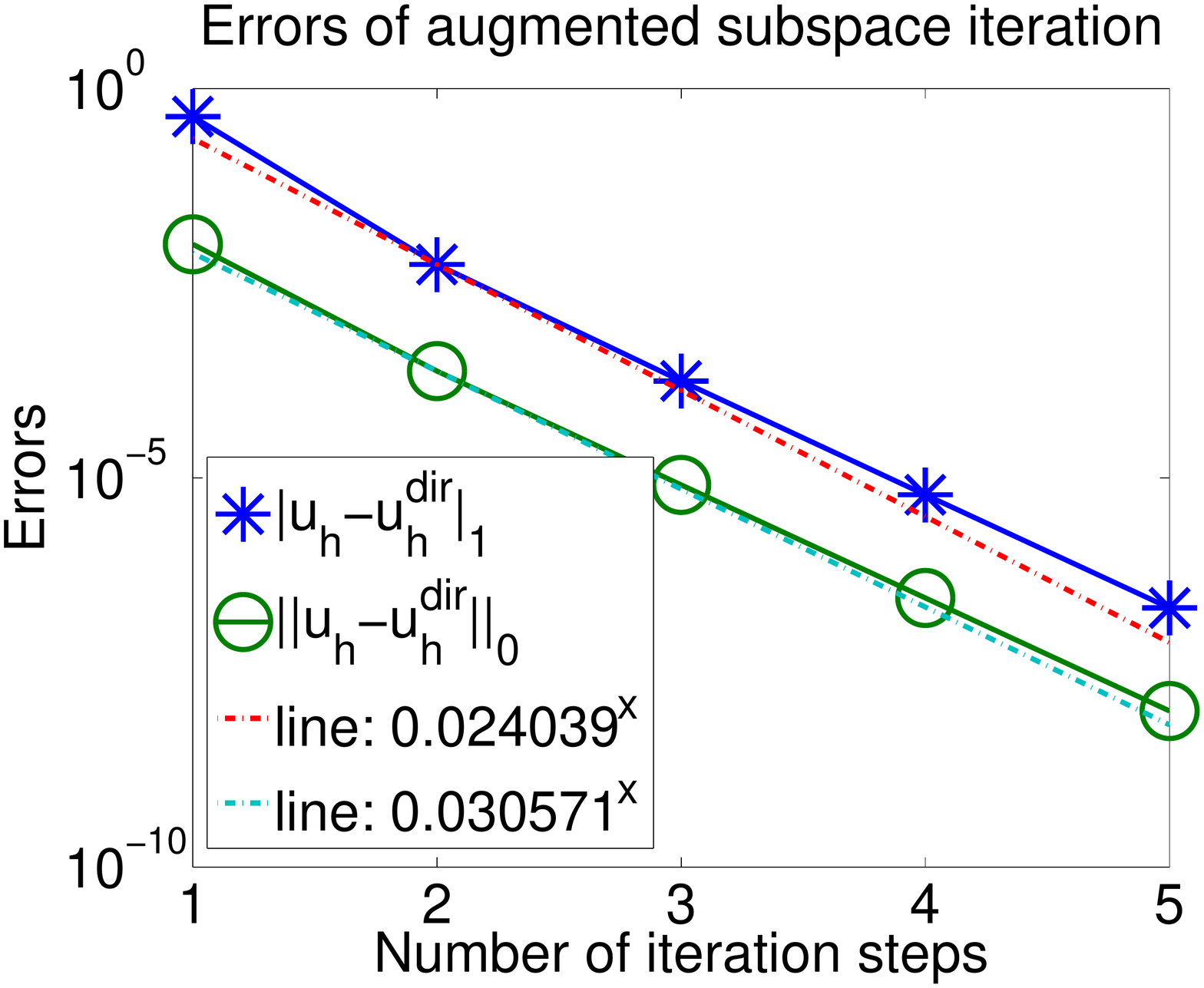}
\includegraphics[width=6cm,height=4.5cm]{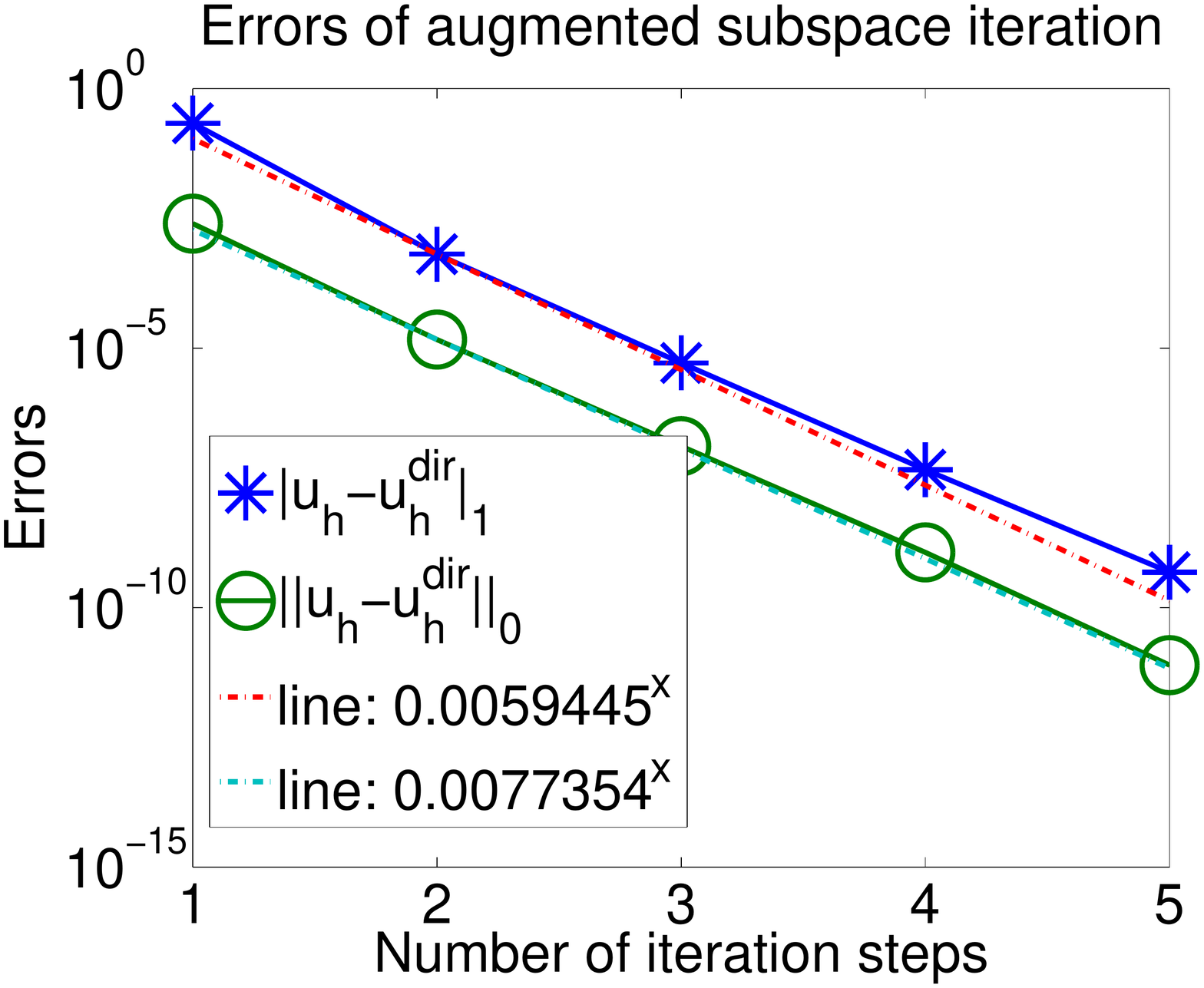}
\caption{The convergence behaviors for the only $4$-th eigenfunction by Algorithm \ref{Algorithm_1}
with the coarse space being the linear finite element space on the mesh with size $H=\sqrt{2}/8$, $\sqrt{2}/16$, $\sqrt{2}/32$ and $\sqrt{2}/64$.
The corresponding convergence rates are $0.33687$, $0.11207$, $0.030571$ and $0.0077354$.}\label{Result_Coarse_Mesh_4_Only}
\end{figure}

\section{Concluding remarks}
In this paper, some enhanced error estimates for the augmented subspace method are deduced for
solving eigenvalue problems. We have derived higher order convergence rates than
existing results. Based on these new results, we can also
produce the corresponding sharper error estimates for the multigrid and multilevel methods
which are designed based on the augmented subspace methods and the sequence of grids.

\end{document}